\documentclass[11pt,reqno]{amsart}
\usepackage[T1]{fontenc}
\usepackage{lmodern,amsmath,amssymb,amsthm,mathtools,microtype}
\usepackage{pgfplots}
\pgfplotsset{compat=1.18}
\usetikzlibrary{patterns}
\usepackage[letterpaper,margin=1in]{geometry}
\usepackage{graphicx}
\usepackage[hidelinks]{hyperref}
\ifdefined\ReviewCopy
  \usepackage[mathlines]{lineno}
\fi
\newtheorem{theorem}{Theorem}[section]
\newtheorem{proposition}[theorem]{Proposition}
\newtheorem{lemma}[theorem]{Lemma}
\newtheorem{corollary}[theorem]{Corollary}
\theoremstyle{definition}
\newtheorem{definition}[theorem]{Definition}
\theoremstyle{remark}
\newtheorem{remark}[theorem]{Remark}
\DeclareMathOperator{\Rea}{Re}
\newcommand{\R}{\mathbb R}
\newcommand{\C}{\mathbb C}
\newcommand{\eps}{\varepsilon}
\newcommand{\dd}{\,dx}
\newcommand{\tht}{\eta}
\numberwithin{equation}{section}
\allowdisplaybreaks[1]
\newcommand{\AuthorNames}{Jeffrey Humpherys}
\newcommand{\AuthorRunningNames}{J. Humpherys}
\newcommand{\AuthorAffiliation}{Missouri University of Science and Technology, Rolla, MO, USA}
\newcommand{\AuthorEmails}{\email{jhumpherys@gmail.com}}
\newcommand{\ConflictStatement}{The author declares no competing interests.}

\hypersetup{pdftitle={Spectral stability for strong ideal-gas shock layers},
 pdfsubject={Spectral stability of strong ideal-gas shock layers},pdfauthor={}}
\title[Spectral stability of strong shock layers]
{Spectral stability for strong ideal-gas shock layers}
\author[\AuthorRunningNames]{\AuthorNames}
\address{\AuthorAffiliation}
\AuthorEmails
\subjclass[2020]{Primary 35B35, 76N06; Secondary 35L67, 35Q30}
\keywords{Compressible Navier--Stokes equations, ideal gas, viscous shock,
 spectral stability, energy estimate, Evans function, strong-shock limit}
\date{November 10, 2026}
\newcommand{\LeanArchive}{is archived at Zenodo \\cite{HumLean}}
\begin{document}
\begin{abstract}
We give a purely analytical proof of one-dimensional spectral stability of sufficiently strong shock layers for the full ideal-gas Navier--Stokes
equations with constant viscosity and heat conduction. Earlier work reduced this question to a limiting problem and gave strong numerical evidence of stability through Evans function computation, but no proof. We prove stability for every adiabatic index $1 <\gamma < 3$ and for an explicit range of ratios of heat conduction to viscosity that contains all common gases.
\end{abstract}
\maketitle
\ifdefined\ReviewCopy\linenumbers\fi

\section{Introduction and main results}\label{sec:introduction}
\subsection{Viscous shock layers}
Shock waves are among the most prominent features of compressible flow. They carry substantial energy across long distances as coherent fronts. In an inviscid gas, a shock takes the form of a discontinuity moving at constant speed between two constant states. Once viscosity and heat conduction are added, this discontinuity is smoothed into a thin traveling wave with the same speed and endstates. This is known as the viscous shock layer. The study of these layers dates back to Becker \cite{Becker}, who derived an explicit profile for a special ratio of heat conduction to viscosity.  For more general fluids, Bethe \cite{Bethe} found the equation-of-state conditions under which shocks behave as in an ideal gas, and Weyl \cite{Weyl} studied the corresponding layers.  Gilbarg \cite{Gilbarg} proved that every finite-amplitude shock in an ideal gas possesses a viscous layer.

Existence alone does not guarantee that a shock layer is physically meaningful. In any real flow, the layer is subjected to small disturbances.  So a shock layer is observable only if it is stable, meaning that these disturbances decay over time and the flow returns to a translate of the original layer, a property known as asymptotic orbital stability.

For the inviscid problem, stability has long been understood. Erpenbeck \cite{Erpenbeck} formulated the stability problem for shock fronts and studied its spectral form, and Majda \cite{Majda} developed the nonlinear theory. In particular, Majda's theorem states that inviscid shocks in a polytropic ideal gas are stable regardless of amplitude, in both one and several space dimensions. The viscous counterpart has proved to be much harder. The inviscid condition is expressed by a Lopatinski determinant computed explicitly from the two endstates, while the viscous condition depends on the entire profile and requires a considerably more sophisticated linear and nonlinear theory.

\subsection{From nonlinear stability to an eigenvalue problem}
For many traveling waves, such as those arising in reaction--diffusion
equations, stability follows once the linearized operator has no
eigenvalues in the closed right half-plane other than a simple
translational eigenvalue at the origin, and the rest of its spectrum is
bounded away from the imaginary axis. Viscous shock layers lack this
spectral gap.  Indeed, for shock layers of conservation laws, the essential spectrum
of the linearized operator touches the imaginary axis at the origin.  As a result, perturbations decay only at algebraic rates, and ruling out unstable eigenvalues no longer gives stability by the usual argument.

A stability theory that works without a spectral gap was developed over the
last several decades, largely by Zumbrun and his collaborators. It is built
on the Evans function, an analytic function of the spectral parameter whose
zeros detect solutions of the eigenvalue equations with the correct behavior
at both spatial infinities \cite{AGJ,GZ,PZ,Zhandbook}. The gap lemma, proved
independently by Gardner and Zumbrun \cite{GZ} and by Kapitula and Sandstede
\cite{KS}, extends the Evans function analytically to a neighborhood of the
origin, where the essential spectrum reaches the imaginary axis. Using
pointwise bounds on the Green function, Zumbrun and Howard \cite{ZH} and
Mascia and Zumbrun \cite{MZ1,MZ2} showed that for Lax shock profiles of
hyperbolic--parabolic systems, including the compressible Navier--Stokes
equations, a spectral stability condition stated through the Evans function
implies asymptotic orbital stability, with no restriction on the amplitude of
the shock. In this way the stability of a given family of shock layers comes
down to a concrete question. Does the Evans function vanish anywhere in the
closed right half-plane?

For weak shocks this question has a satisfactory answer. Energy methods of
Kawashima and Matsumura \cite{KM}, Goodman \cite{Goodman}, Liu \cite{Liu}, and Szepessy and Xin \cite{SX}
established stability of shock profiles of small amplitude, and Humpherys
and Zumbrun \cite{HuZ}, Plaza and Zumbrun \cite{PZ}, and Freist\"uhler and
Szmolyan \cite{FS} proved spectral stability of small-amplitude profiles
for broad classes of systems, including the full Navier--Stokes equations.
All of these arguments use the small variation of the wave. For shocks of
large amplitude, analytical results have been rare.

\subsection{The strong-shock limit}
The large-amplitude problem has therefore mostly been studied
numerically. One computes the Evans function on a large semicircle in the
right half-plane that encloses all possible unstable eigenvalues and counts
its winding number, which by the argument principle equals the number of
unstable eigenvalues. This decides the stability of individual waves, or of
compact families of waves, with high confidence. The limit of infinitely
strong shocks, however, is not reached by any finite computation.

For the full ideal-gas equations, Humpherys, Lyng, and Zumbrun \cite{HLZ}
showed that this limit is nevertheless tractable. After a suitable rescaling,
the density stays bounded and bounded away from zero for every shock
strength, and only the upstream internal energy degenerates, tending to zero.
The strong-shock limit is then a regular perturbation of the profile
equations, and although consistent splitting is lost upstream on the
imaginary axis, the Evans function is continuous in all parameters up to and
including the limit. Together with a uniform bound on the size of unstable
eigenvalues, this makes the shocks and their limits a compact family of
boundary-value problems. Winding-number computations over this family in
\cite{HLZ} found stability for adiabatic indices between $1.2$ and $3$ and
for a broad range of ratios of heat conduction to viscosity, including the
values for air. The multidimensional problem was studied numerically in
\cite{HLZmulti}.

These computations were not proofs, and \cite{HLZ} left their rigorous
verification as an open problem. A computer-assisted proof in interval
arithmetic is possible in principle, as Barker showed for a related problem
\cite{Barker}, but it requires rigorous bounds on numerically computed
profiles and Evans functions over a continuum of parameters. The
reduction of \cite{HLZ} also points to a second route. Because the Evans
function is continuous up to the strong-shock limit and unstable
eigenvalues are uniformly bounded, the stability of all sufficiently
strong shocks follows from a single statement about the limiting problem,
namely that the limiting Evans function has no zero in the closed right
half-plane. If that statement can be proved by analysis, the
one-dimensional stability of strong ideal-gas shocks becomes a theorem
with no numerical step.

\subsection{Contribution of this paper}
This paper carries out that program. We prove that the limiting Evans
function of \cite{HLZ} does not vanish in the closed right half-plane for
every adiabatic index $1<\gamma<3$ and an explicit range of transport ratios
depending on $\gamma$, which always contains the interval $[2/3,9/5]$. The region contains the
values for air, water vapor, and other common gases, the kinetic-theory value
for monatomic gases, and the line of
transport ratios on which the limiting profile is the strong-shock limit of
Becker's profile. Combined with the convergence and high-frequency results
of \cite{HLZ}, this gives spectral stability of all sufficiently strong
finite shocks in this region, uniformly on compact sets of parameters. No
step of the argument uses numerical integration of a profile, numerical
evaluation of an Evans function, or numerical verification of an
inequality.

The main obstacle to an energy estimate is the partial diffusion of the
equations. Momentum and thermal energy diffuse, but mass does not, and in
the integrated eigenvalue equations the thermal equation is coupled to the
derivative of the mass perturbation. That coupling cannot be controlled by
the diffusion. The key step is a profile-dependent change of the integrated
thermal variable that removes this coupling exactly, for every adiabatic
index and transport ratio. What remains is a weighted energy estimate.
Measuring the velocity perturbation in units of the translation mode
factors the momentum operator and removes the zeroth-order coupling between
velocity and heat. A shift of the thermal derivative by a pressure term
removes a profile derivative from the thermal potential, and positivity
reduces to a two-by-two matrix along the profile. For
$1\le\nu\le\Gamma+1$, a constant thermal weight suffices, and an invariant interval
for a normalized slope of the profile controls the matrix directly.
Elsewhere, a thermal weight that follows the profile writes the matrix as an
interpolation of two endpoint matrices plus positive remainders. The
determinant of the first endpoint matrix increases with the slope of the
profile, so its sign at the downstream slope is sufficient for positivity,
and this sufficient condition defines the region. No formula for the profile is needed, and every
inequality is proved by hand. Zero frequency, where the energy estimate
degenerates, is handled by a separate argument based on the variational
equation of the profile and the outgoing condition continued from positive
frequencies.

The arguments of this paper have been checked in the Lean proof assistant,
with the results of \cite{HLZ} that they use entering as hypotheses.
Section~\ref{sec:formal} describes what was checked.

\subsection{Equations and normalization}
We now set up the problem precisely. In Lagrangian coordinates, the
Navier--Stokes equations for a viscous, heat-conducting ideal gas read
\[
 v_t-u_x=0,\qquad
 u_t+p_x=\left(\frac{\mu_{\mathrm{eff}} u_x}{v}\right)_x,\qquad
 E_t+(pu)_x=\left(\frac{\mu_{\mathrm{eff}} uu_x}{v}\right)_x
 +\left(\frac{k\, T_x}{v}\right)_x,
\]
with $E=e+u^2/2$ and the ideal-gas laws $p=\Gamma e/v$ and $e=c_vT$, where $\Gamma$ is the Gr\"uneisen coefficient introduced below. A viscous shock profile is a
traveling-wave solution moving with speed $s$ and connecting constant
states at $x=\pm\infty$. The ideal-gas laws are invariant under the scaling
\[
 (x,t,v,u,T)\mapsto
 \left(-\epsilon sx,\ \epsilon s^2t,\ \frac v\epsilon,\
 -\frac u{\epsilon s},\ \frac{T}{\epsilon^2s^2}\right),
\]
with $\epsilon$ chosen so that the upstream specific volume is one
\cite[Section 2.2]{HLZ}. After this rescaling, a further dilation of $x$
and $t$ that normalizes the viscosity, and a change to the frame in which
the shock is stationary, the equations become
\begin{align}
 v_t+v_x-u_x&=0,\label{eq:pde1}\\
 u_t+u_x+p_x
   &=\left(\frac{u_x}{v}\right)_x,\label{eq:pde2}\\
 E_t+E_x+(pu)_x
   &=\left(\frac{u u_x+\nu e_x}{v}\right)_x,
 \qquad E=e+\frac{u^2}{2},\quad
 p=\frac{\Gamma e}{v}.\label{eq:pde3}
\end{align}
Here $v$ is specific volume, $u$ is velocity, and $e$ is specific
internal energy. These are the rescaled equations of
\cite[Section 2]{HLZ}. The Gr\"uneisen coefficient is $\Gamma=\gamma-1>0$, where
$\gamma=c_p/c_v$ is the adiabatic index. The effective longitudinal
viscosity $\mu_{\mathrm{eff}}>0$ has been normalized to one, and the
remaining transport ratio is
\begin{equation}\label{eq:physicalnu}
 \nu=\frac{k}{c_v\mu_{\mathrm{eff}}}>0,
\end{equation}
where $k>0$ is the heat conductivity. Thus our $\nu$ is the ratio
$\nu/\mu$ in the unnormalized notation of \cite{HLZ}. The coefficients and
specific heats are constant throughout this paper.

For physical interpretation, let $\mu_s$ denote shear viscosity and
$\operatorname{Pr}=c_p\mu_s/k$ the usual Prandtl number.
Under Stokes' hypothesis, $\mu_{\mathrm{eff}}=4\mu_s/3$, so
\begin{equation}\label{eq:Prandtl}
 \nu=\frac{3(\Gamma+1)}{4\operatorname{Pr}}.
\end{equation}
This is the conversion in \cite[Appendix A]{HLZ}. Stokes' hypothesis is
used only for this physical interpretation, not for the spectral theorems,
which are stated in terms of $\Gamma$ and $\nu$.

Write $\hat v$ and $\hat e$ for the profile specific volume and
internal energy. After the rescaling the shock speed is fixed, the
upstream state is $(1,0,e_-)$, and the profile velocity is $\hat u=\hat v-1$.
The Rankine--Hugoniot relations then determine every endstate quantity
from the downstream specific volume $v_+$ \cite[Section 2.4]{HLZ},
\begin{equation}\label{eq:endstates}
 v_*:=\frac{\Gamma}{\Gamma+2},\qquad
 a_*:=1-v_*=\frac{2}{\Gamma+2},\qquad
 e_- =\frac{(\Gamma+2)v_+-\Gamma}{2\Gamma(\Gamma+1)}.
\end{equation}
A physical nonconstant shock has $v_*<v_+<1$. The value $v_+=1$ is the
characteristic limit of vanishing amplitude, and the other end of the
range is the strong-shock limit $v_+\to v_*$. In terms of the upstream
Mach number $M$, the strong-shock limit is $M\to\infty$; see
\eqref{eq:Mach}. In this limit the upstream internal energy $e_-$ tends to
zero, while the specific volume stays between $v_*$ and one and the
downstream energy tends to $e_+=a_*^2/2$. The profile equations remain
smooth, indeed polynomial, at $e_-=0$ \cite[Remark 2.1]{HLZ}, so the
limiting endstates still define a traveling wave. We refer to it as the
\emph{limiting profile}. When $\nu=\Gamma+1$, it is the strong-shock limit
of the classical Becker profile.

\subsection{The spectral problem}
Linearizing \eqref{eq:pde1}--\eqref{eq:pde3} about a profile gives the eigenvalue problem
for the shock layer. As noted above, it has an eigenvalue at the origin
coming from translation, and its essential spectrum reaches the origin. We
therefore follow \cite{HLZ} and work with integrated perturbation variables. Because each equation is in
conservation form, the integrated perturbations satisfy a closed linear
system. Integration removes the translational eigenvalue while keeping
analytic dependence on $\lambda$, and for $\lambda\ne0$ it does not change
the eigenvalues \cite{ZH,HLZ}. Section~\ref{sec:spectral} derives this system and
describes its behavior at spatial infinity.

Write $D(\lambda;\Gamma,\nu,v_+)$ for the integrated Evans function of
the shock with parameters $\Gamma$, $\nu$, and $v_+$, with continuous
choices of normalization as in \cite{HLZ}, and define
\begin{equation}\label{eq:Dstar}
 D_*(\lambda;\Gamma,\nu)
 :=D(\lambda;\Gamma,\nu,v_*(\Gamma)).
\end{equation}
The incoming and outgoing solution spaces are continued from the
physical spectral problem. Their precise interpretation, especially at
zero frequency, is given in Section~\ref{sec:spectral}. Spectral stability
of a shock means that $D$ has no zero in the closed right half-plane
$\Rea\lambda\ge0$. By the results of \cite{HLZ} recalled above, the
spectral stability of all sufficiently strong shocks follows from the same
statement for $D_*$, and this is the statement we prove.

\subsection{Main results}
For $\nu>0$ consider the quadratic polynomial
\begin{equation}\label{eq:Pi}
 \Pi_\nu(\rho)=9\bigl(32\nu-9(\nu-1)^2\bigr)\bigl(3\rho^2+4\rho-\nu(\nu-2)\bigr)
 -\bigl((11\nu-9)\rho-9\nu(\nu-1)\bigr)^2 .
\end{equation}
For $9/5<\nu<5$ it has exactly one zero $\rho_E(\nu)$ in the interval
$0<\rho<\nu-1$ (Lemma~\ref{lem:det-monotone}), and we set $\rho_E(\nu)=0$ for
$\nu\le9/5$. Let
\begin{equation}\label{eq:GammaE}
 \Gamma_E(\nu)=\frac{2\nu\rho_E}{2\nu+(\nu-1)\rho_E-\rho_E^2},
 \qquad \rho_E=\rho_E(\nu),
\end{equation}
and
\begin{equation}\label{eq:region}
 \mathcal R=\Bigl\{(\Gamma,\nu):\ 0<\Gamma<2,\quad \frac23<\nu<5,\quad
 \Gamma\ge\Gamma_E(\nu)\Bigr\}.
\end{equation}
\begin{theorem}[Limiting stability]\label{thm:main}
For every $(\Gamma,\nu)\in\mathcal R$,
\[
 D_*(\lambda;\Gamma,\nu)\ne0\qquad\text{for all }\Rea\lambda\ge0.
\]
\end{theorem}

The convergence and high-frequency results of \cite{HLZ} then transfer
this limiting statement to finite shocks.

\begin{corollary}[Sufficiently strong finite shocks]\label{cor:finite}
For every compact set $\mathcal K\subset\mathcal R$ there exists
$\delta_{\mathcal K}>0$ such that the physical shocks satisfying
\[
 (\Gamma,\nu)\in\mathcal K,\qquad 0<v_+-v_*(\Gamma)<\delta_{\mathcal K}
\]
satisfy the integrated Evans stability condition.
\end{corollary}

\subsection{Discussion of the results and method}
The condition $\Gamma<2$ is $\gamma<3$, so the theorem covers every
adiabatic index of an ideal gas. The conditions on $\nu$ have a direct
meaning in the proof. On the interval $1\le\nu\le\Gamma+1$ we use a constant
thermal weight, and elsewhere a weight that follows the profile. With the
second weight, the energy matrix is an interpolation of two endpoint
matrices plus nonnegative remainders, and $\Pi_\nu$ is the determinant of the
first endpoint matrix up to a positive factor. This determinant increases
with the slope ratio of the profile, so its nonnegativity at the downstream
slope ratio is sufficient for coercivity. The value $\Gamma_E(\nu)$ is the
one at which the downstream slope ratio equals $\rho_E(\nu)$, so
$\Gamma\ge\Gamma_E(\nu)$ is exactly the condition that this determinant is
nonnegative there. The condition is sufficient, not necessary, and the upper
boundary of $\mathcal R$ is a limitation of this estimate, not a threshold at
which coercivity is known to fail. The bound $\nu>2/3$ comes from the decay
of the variable weight at the upstream state (Lemma~\ref{lem:boundary}), and
$\nu<5$ keeps the coefficient $1-\kappa_0$ of $|\Phi|^2$ in the first
endpoint matrix above one tenth.

The function $\Gamma_E$ vanishes for $\nu\le9/5$ and satisfies
$\Gamma_E(\nu)<\nu-1$, so $\mathcal R$ contains the strip $2/3<\nu\le9/5$ and
the reference line $\nu=\Gamma+1$ for all $0<\Gamma<2$, and the reference line
lies in its interior. Corollary~\ref{cor:finite} therefore applies to every
compact family of transport ratios near it. Let $\nu_E(\Gamma)$ denote
the largest $\nu$ with $(\Gamma,\nu)\in\mathcal R$. Numerical evaluation of
\eqref{eq:GammaE}, which the proof does not use, indicates that $\Gamma_E$
increases with $\nu$ and reaches $2$ near $\nu\approx4.95$. If so, for each
$\Gamma$ the region is the interval $2/3<\nu\le\nu_E(\Gamma)$, and under the
physical convention in \eqref{eq:Prandtl} it becomes
\begin{equation}\label{eq:physicalregion}
 \frac{3(1+\Gamma)}{4\nu_E(\Gamma)}\le\operatorname{Pr}
 <\frac{9(1+\Gamma)}8 .
\end{equation}
The values of $\nu_E$ quoted here and in Appendix~\ref{app:physical} are
computed numerically from \eqref{eq:GammaE}. At the diatomic value
$\Gamma=2/5$ one has $\nu_E\approx2.47$ and
$0.43\le\operatorname{Pr}<1.57$, and at the monatomic value
$\Gamma=2/3$ one has $\nu_E\approx2.94$ and
$0.43\le\operatorname{Pr}<1.87$, with the lower ends rounded up and the upper
ends rounded down. In particular, the air ratio $\nu\approx1.47$ reported in
\cite[Appendix A]{HLZ} and the monatomic kinetic-theory value
$\operatorname{Pr}=2/3$, corresponding to $\nu=15/8$, lie in the
respective intervals; see also \cite[Appendix B]{HLZ}.
Appendix~\ref{app:physical} compares the region with representative gases,
and Figure~\ref{fig:regions} shows the region. These statements
concern the constant-coefficient ideal-gas model, not temperature-dependent
transport or additional physical effects. The finite-shock thresholds are
existential, and no uniform amplitude threshold is asserted as
$\Gamma\downarrow0$.

\begin{figure}[tbp]
 \centering
 \begin{tikzpicture}
 \begin{axis}[width=.82\textwidth,height=.58\textwidth,xmin=0,xmax=2,ymin=0.6,ymax=5.1,
  xlabel={$\Gamma=\gamma-1$},ylabel={$\nu$},tick label style={font=\small},
  legend style={at={(0.02,0.98)},anchor=north west,draw=none,fill=none,font=\scriptsize},
  legend cell align=left]
 \addplot[fill=black!15,draw=none,forget plot] coordinates {(0,1) (2,1) (2,3) (0,1)};
 \addplot[fill=black!35,draw=none,forget plot] coordinates {(0,1)
 (0.000,1.8000) (0.040,1.8662) (0.080,1.9324) (0.120,1.9987) (0.160,2.0652) 
 (0.200,2.1321) (0.240,2.1994) (0.280,2.2671) (0.320,2.3352) (0.360,2.4037) 
 (0.400,2.4727) (0.440,2.5422) (0.480,2.6121) (0.520,2.6824) (0.560,2.7532) 
 (0.600,2.8243) (0.640,2.8957) (0.680,2.9674) (0.720,3.0394) (0.760,3.1115) 
 (0.800,3.1838) (0.840,3.2561) (0.880,3.3285) (0.920,3.4007) (0.960,3.4728) 
 (1.000,3.5446) (1.040,3.6160) (1.080,3.6870) (1.120,3.7574) (1.160,3.8271) 
 (1.200,3.8961) (1.240,3.9641) (1.280,4.0311) (1.320,4.0969) (1.360,4.1614) 
 (1.400,4.2246) (1.440,4.2863) (1.480,4.3463) (1.520,4.4046) (1.560,4.4611) 
 (1.600,4.5157) (1.640,4.5683) (1.680,4.6189) (1.720,4.6674) (1.760,4.7138) 
 (1.800,4.7580) (1.840,4.8001) (1.880,4.8401) (1.920,4.8779) (1.960,4.9136) 
 (2.000,4.9473) 
  (2,3) (0,1)};
 \addplot[fill=black!35,draw=none,forget plot] coordinates {(0,0.6667) (2,0.6667) (2,1) (0,1)};
 \addplot[black,thin,forget plot] coordinates {
 (0.000,1.8000) (0.040,1.8662) (0.080,1.9324) (0.120,1.9987) (0.160,2.0652) 
 (0.200,2.1321) (0.240,2.1994) (0.280,2.2671) (0.320,2.3352) (0.360,2.4037) 
 (0.400,2.4727) (0.440,2.5422) (0.480,2.6121) (0.520,2.6824) (0.560,2.7532) 
 (0.600,2.8243) (0.640,2.8957) (0.680,2.9674) (0.720,3.0394) (0.760,3.1115) 
 (0.800,3.1838) (0.840,3.2561) (0.880,3.3285) (0.920,3.4007) (0.960,3.4728) 
 (1.000,3.5446) (1.040,3.6160) (1.080,3.6870) (1.120,3.7574) (1.160,3.8271) 
 (1.200,3.8961) (1.240,3.9641) (1.280,4.0311) (1.320,4.0969) (1.360,4.1614) 
 (1.400,4.2246) (1.440,4.2863) (1.480,4.3463) (1.520,4.4046) (1.560,4.4611) 
 (1.600,4.5157) (1.640,4.5683) (1.680,4.6189) (1.720,4.6674) (1.760,4.7138) 
 (1.800,4.7580) (1.840,4.8001) (1.880,4.8401) (1.920,4.8779) (1.960,4.9136) 
 (2.000,4.9473) };
 \addplot[black,thin,forget plot] coordinates {(0,0.6667) (2,0.6667)};
 \addplot[black,thin,dashed,forget plot] coordinates {(0,1) (2,3)};
 \addlegendimage{area legend,fill=black!15,draw=none}
 \addlegendentry{constant weight}
 \addlegendimage{area legend,fill=black!35,draw=none}
 \addlegendentry{variable weight}
 \addplot[black,densely dashed,very thin] coordinates {
 (0.005,1.0896) (0.040,1.2841) (0.080,1.4317) (0.120,1.5580) (0.160,1.6735) 
 (0.200,1.7821) (0.240,1.8859) (0.280,1.9860) (0.320,2.0830) (0.360,2.1774) 
 (0.400,2.2696) (0.440,2.3597) (0.480,2.4480) (0.520,2.5347) (0.560,2.6197) 
 (0.600,2.7032) (0.640,2.7853) (0.680,2.8660) (0.720,2.9454) (0.760,3.0235) 
 (0.800,3.1004) (0.840,3.1760) (0.880,3.2505) (0.920,3.3238) (0.960,3.3960) 
 (1.000,3.4671) (1.040,3.5371) (1.080,3.6061) (1.120,3.6740) (1.160,3.7409) 
 (1.200,3.8068) (1.240,3.8717) (1.280,3.9356) (1.320,3.9986) (1.360,4.0607) 
 (1.400,4.1219) (1.440,4.1822) (1.480,4.2416) (1.520,4.3002) (1.560,4.3579) 
 (1.600,4.4148) (1.640,4.4709) (1.680,4.5261) (1.720,4.5806) (1.760,4.6343) 
 (1.800,4.6873) (1.840,4.7395) (1.880,4.7910) (1.920,4.8418) (1.960,4.8919) 
 (1.999,4.9400) };
 \addlegendentry{constant-weight limit $\nu=\nu_D(\Gamma)$}
 \addplot[black,densely dotted,thick] coordinates {(0,0.75) (0.6667,1.875)};
 \addlegendentry{Eucken estimate}
 \addplot[only marks,mark=o,mark size=1.5pt,black,forget plot] coordinates
  {(0.6667,1.875) (0.40,1.48) (0.29,1.26) (0.33,1.00)};
 \draw[very thin] (axis cs:0.6667,1.875) -- (axis cs:1.0,1.62)
  node[font=\scriptsize,anchor=west] {monatomic};
 \draw[very thin] (axis cs:0.40,1.48) -- (axis cs:0.95,1.36)
  node[font=\scriptsize,anchor=west] {air, H$_2$};
 \draw[very thin] (axis cs:0.29,1.26) -- (axis cs:0.90,1.13)
  node[font=\scriptsize,anchor=west] {CO$_2$};
 \draw[very thin] (axis cs:0.33,1.00) -- (axis cs:0.60,0.82)
  node[font=\scriptsize,anchor=west] {H$_2$O};
 \node[font=\scriptsize,rotate=24] at (axis cs:1.15,4.05) {$\nu=\nu_E(\Gamma)$};
 \node[font=\scriptsize,rotate=17] at (axis cs:1.6,2.45) {$\nu=\Gamma+1$};
 \end{axis}
 \end{tikzpicture}
 \caption{The region $\mathcal R$ of Theorem~\ref{thm:main} in the
 $(\Gamma,\nu)$ plane. For $1\le\nu\le\Gamma+1$ the proof uses a
 constant thermal weight, and elsewhere a weight that follows the
 profile. The solid curve $\nu=\nu_E(\Gamma)$ is the upper boundary of
 $\mathcal R$, and the thin dashed curve $\nu=\nu_D(\Gamma)$ is the limit of
 the constant weight (Remark~\ref{rem:sharp}), both computed numerically.
 The dotted curve is Eucken's estimate of the transport ratio, and the marked
 points are the gases of Appendix~\ref{app:physical}.}
 \label{fig:regions}
\end{figure}

When $\nu=\Gamma+1$, the limiting profile is the strong-shock limit of
the Becker profile \cite{Becker,Johnson}. By \eqref{eq:Prandtl},
this is the case $\operatorname{Pr}=3/4$ under Stokes' hypothesis.
Its explicit form provides a useful reference family, but the coercivity
argument does not use it. For general transport ratios, the invariant
interval for the normalized profile slope supplies all the information
about the profile that the estimate needs. Appendix~\ref{sec:continuation}
records a separate argument by continuation in the parameters, which is
not needed for the main result.

\begin{remark}[Meaning of spectral stability]\label{rem:spectrum}
The theorem and its corollary state the integrated Evans stability condition. The
translational eigenvalue of the unintegrated problem is not an Evans zero
that must be excluded by the nonzero-frequency energy estimate. In the
strong-shock limit, the upstream equations have a neutral transport branch,
so the conclusion is not a spectral-gap assertion for the limiting
operator. At $\lambda=0$, the continued outgoing space contains specified
bounded modes, and ruling out only a two-sided decaying kernel would not
prove the theorem.
\end{remark}

\subsection{Organization of the paper}
Section~\ref{sec:profile} describes the limiting profile and proves the
slope enclosure. Section~\ref{sec:spectral} derives the integrated
eigenvalue equations and the Evans function, and Sections~\ref{sec:thermal}
and~\ref{sec:energy} give the change of thermal variable and the energy
identity. Section~\ref{sec:coercivity} proves coercivity on $\mathcal R$.
Sections~\ref{sec:admissible} and~\ref{sec:zero} justify admissibility and
establish limiting spectral stability. Section~\ref{sec:finite} passes to
sufficiently strong finite shocks, and Section~\ref{sec:formal} describes
the formal verification. Appendix~\ref{app:identity}
records the full divergence identity, Appendix~\ref{sec:continuation} gives
an Evans continuation argument, and Appendix~\ref{app:physical} compares the region with physical transport
ratios.

\section{The limiting profile}\label{sec:profile}
Integrating the stationary equations from the upstream state gives
$\hat u=\hat v-1$ and
\begin{align}
 \hat v'&=\hat v(\hat v-1)+\Gamma(\hat e-\hat v e_-),\label{eq:finiteprofile1}\\
 \hat e'&=\frac{\hat v}\nu\left(\hat e-e_--\frac{(\hat v-1)^2}{2}
                      +\Gamma e_-(\hat v-1)\right).\label{eq:finiteprofile2}
\end{align}
At $e_-=0$, these reduce to
\begin{equation}\label{eq:profile}
 \hat v'=\hat v(\hat v-1)+\Gamma \hat e,\qquad
 \hat e'=\frac{\hat v}\nu\left(\hat e-\frac{(\hat v-1)^2}{2}\right).
\end{equation}
The boundary conditions are
\[
 (\hat v,\hat e)(-\infty)=(1,0),\qquad
 (\hat v,\hat e)(+\infty)=(v_*,a_*^2/2).
\]
The existence, uniqueness up to translation, and exponential endstate
convergence of the profile family up to the strong-shock limit follow from
\cite[Section 2.5]{HLZ}. For the transport relation
$\nu=\Gamma+1$, the strong-shock limit of the Becker profile
\cite{Becker} has the following explicit form in Lagrangian coordinates.
For a modern derivation of the Becker family, see \cite[Section 3.1]{Johnson}.

\begin{lemma}[Explicit limiting profile]\label{lem:profile}
Suppose $\Gamma>0$ and $\nu=\Gamma+1$. Put $\gamma=\Gamma+1$.
For an arbitrary translation $x_0\in\R$, define
\begin{equation}\label{eq:explicitprofile}
 s(x)=\frac1{1+\exp((x-x_0)/\gamma)},\qquad
 \hat v=\frac{\Gamma+2s}{\Gamma+2},\qquad
 \hat e=\frac{1-\hat v^2}{2\gamma}.
\end{equation}
These functions solve \eqref{eq:profile} with the stated boundary
conditions. With $a=1-\hat v=a_*(1-s)$, they satisfy
\begin{equation}\label{eq:profileidentities}
 s'=-\frac{s(1-s)}\gamma,\qquad
 \hat v'=-\frac{as}\gamma,\qquad
 \hat e-\frac{a^2}{2}=\frac{as}\gamma.
\end{equation}
In particular, $0<s<1$, $a>0$, and $\hat v'<0$ at every finite $x$.
\end{lemma}
\begin{proof}
Direct differentiation gives \eqref{eq:profileidentities} and
$\hat e'=-\hat v\hat v'/\gamma$. These identities imply \eqref{eq:profile} when
$\nu=\gamma$. The endstates follow by taking $s\to1$ and $s\to0$ in
\eqref{eq:explicitprofile}.
\end{proof}

We use $a=1-\hat v$ and $s=(\hat v-v_*)/a_*$ throughout the limiting-profile
analysis. Thus $a$ is the local volume drop and $s$ is a dimensionless
profile coordinate; neither denotes a spectral unknown. From
\eqref{eq:explicitprofile},
\begin{equation}\label{eq:profiledecay}
 a(x)\sim a_*\,e^{(x-x_0)/\gamma}\quad(x\to-\infty),\qquad
 s(x)\sim e^{-(x-x_0)/\gamma}\quad(x\to+\infty).
\end{equation}
Here $\sim$ means that the ratio of the two sides tends to one.

\subsection{A slope enclosure for general transport parameters}
The explicit relation between $\hat e$ and $\hat v$ is not needed. A bound
on the profile slope will suffice.

\begin{lemma}[Invariant slope interval]\label{lem:slopestrip}
For $\Gamma>0$ and $\nu>0$, the limiting profile satisfies
\begin{equation}\label{eq:zstrip}
 z:=\frac{-\hat v'}{as}>0,\qquad
 \min\left\{\frac1{\Gamma+1},\frac1\nu\right\}
 \le z\le
 \max\left\{\frac1{\Gamma+1},\frac1\nu\right\}.
\end{equation}
In particular $\hat v'<0$ on the nonconstant profile. If $\nu>\Gamma+1$,
then more precisely
\begin{equation}\label{eq:zsharp}
 \frac1\nu<z<r<\frac1{\Gamma+1}
\end{equation}
at every finite point, where $r=\lim_{x\to+\infty}z$ is the downstream
slope, the positive root of the quadratic $f_0$ in the proof.
If $\nu<1$, then $r<z<1$ at every finite point. In every case
$z\to\min\{1,1/\nu\}$ and $d=sz\to\min\{1,1/\nu\}$ as $x\to-\infty$.
\end{lemma}
\begin{proof}
Eliminating $\hat e$ from the profile equations gives
\[
 s'=-s(1-s)z,\qquad z'=f(s,z),
\]
where
\begin{align*}
 f(s,z)&=(1-s)f_0(z)+sf_1(z),\\
 f_0(z)&=z^2+
 \left(\frac{\Gamma-2}{\Gamma+2}
       +\frac{\Gamma}{\nu(\Gamma+2)}\right)z
       -\frac{\Gamma}{\nu(\Gamma+2)},\\
 f_1(z)&=-(z-1)(z-1/\nu).
\end{align*}
The polynomial $f_0$ has one positive root $r$, and $f_0'(r)>0$.
At $(s,z)=(0,r)$ the spatial eigenvalues are $-r$ and $f_0'(r)$.
Thus the downstream profile is its one-dimensional stable branch into
$s>0$. Set $z_0=1/(\Gamma+1)$ and $z_1=1/\nu$. Direct evaluation gives
\[
 f_0(z_0)=\frac{\Gamma^2(\nu-\Gamma-1)}
 {\nu(\Gamma+1)^2(\Gamma+2)},\qquad
 f_0(z_1)=\frac{2(\Gamma+1-\nu)}{\nu^2(\Gamma+2)}.
\]
These values place $r$ between $z_0$ and $z_1$. If $\nu<\Gamma+1$,
then $z_0<z_1$, $f(s,z_0)\le0$, and $f(s,z_1)\ge0$ for
$0\le s\le1$. In reverse spatial time, both horizontal boundaries
therefore point into $[z_0,z_1]$. If $\nu>\Gamma+1$, the interval
and the signs reverse, with the same conclusion. At equality the interval
collapses to the invariant solution $z=1/(\Gamma+1)$.

For unequal endpoints, the stable branch starts in the interior of the
interval and remains there under reverse flow. Since $z$ has a positive
lower bound, $s$ increases toward one in reverse time and reaches neither
endpoint in finite time. Reconstruction gives
\[
 \hat v=v_*+a_*s,\qquad \hat e=\frac a\Gamma(\hat v-sz).
\]
As $s\to1$, this tends to $(1,0)$; as $s\to0$, it tends to the
required downstream state. Thus the branch is the limiting profile and
satisfies \eqref{eq:zstrip}.

Suppose now $\nu>\Gamma+1$, so that $1/\nu<r<1/(\Gamma+1)$. At a point of
the profile with $z\le1/\nu$ one has $f_0(z)<0$ and $f_1(z)\le0$, hence
$z'<0$, and $z$ then decreases for all later $x$, which is incompatible with
$z\to r$ as $x\to+\infty$. At a point with $z\ge r$ one has $f_0(z)\ge0$ and
$f_1(z)>0$, hence $z'>0$, with the same contradiction. This proves
\eqref{eq:zsharp}.

Suppose $\nu<1$. On the line $z=1$ one has
$f(s,1)=(1-s)f_0(1)=2\Gamma(1-s)/(\Gamma+2)>0$, so in reverse spatial time
this line also points into the region below it, and the profile stays in
$[z_0,1]$, which contains $r$ because $f_0(z_0)<0<f_0(1)$. At a point with
$z\le r$ one has $f_0(z)\le0$ and $f_1(z)<0$, hence $z'<0$, and $z$ then
decreases for all later $x$, which is again incompatible with $z\to r$.
Thus $r<z<1$.

Finally, as $x\to-\infty$ we have $s\to1$, so $z'=f_1(z)+o(1)$. Let $I$ be
the interval established above for $z$, namely $[1/\nu,r]$ if $\nu>\Gamma+1$,
\eqref{eq:zstrip} if $1\le\nu\le\Gamma+1$, and $[r,1]$ if $\nu<1$. On any closed
part of $I$ that avoids the zeros of $f_1$, this derivative has
a fixed sign and stays away from zero sufficiently far upstream, so $z$ tends
to a zero of $f_1$ in the interval. The zeros are $1$ and $1/\nu$. For
$\nu>1$ the interval excludes $1$, for $\nu<1$ it excludes $1/\nu$, and for
$\nu=1$ the two coincide. Hence $z\to\min\{1,1/\nu\}$, and $d=sz$ has the same
limit.
\end{proof}

\section{The integrated spectral problem}\label{sec:spectral}
This section sets up the eigenvalue problem studied in the rest of the
paper; see \cite{HLZ} for details.

\subsection{Why integrated coordinates}
Linearizing \eqref{eq:pde1}--\eqref{eq:pde3} about a profile gives a
linear operator $L$ acting on perturbations of $(v,u,E)$, and the spectral
problem is $L w=\lambda w$. Two features of this problem rule out a
spectral gap. Since every translate of a profile is again a profile, the
derivative of the profile satisfies $L\hat U'=0$, so $\lambda=0$ is an
eigenvalue. Since the endstates are constant states of a system of
conservation laws, the essential spectrum of $L$ reaches the imaginary
axis at the origin \cite{ZH,Zhandbook}. Stability is therefore expressed
through an Evans function rather than through the location of the
spectrum alone.

The equations are in conservation form, and this allows one to pass to
integrated variables. If $(v,u,E)$ is a perturbation, put
\[
 \tilde v(x)=\int_{-\infty}^x v(y)\,dy,\qquad
 \tilde u(x)=\int_{-\infty}^x u(y)\,dy,\qquad
 \tilde E(x)=\int_{-\infty}^x E(y)\,dy .
\]
Integrating each equation from $-\infty$ to $x$ gives a closed linear
system for the integrated perturbations. For $\lambda\ne0$, an eigenfunction
of $L$ integrates to an eigenfunction of the integrated operator, and an
eigenfunction of the integrated operator differentiates to one of $L$, so
the two problems have the same nonzero eigenvalues. At $\lambda=0$, the
profile derivative integrates to the profile perturbation itself, which
does not decay at $+\infty$, so the translational eigenvalue is removed.
The integrated Evans function is analytic in $\lambda$ and is not forced to
vanish at the origin. For finite shocks, its nonvanishing at the origin
corresponds to the low-frequency condition of \cite{GZ,ZH}, which
combines transversality of the profile with one-dimensional inviscid
stability of the shock. These properties make integrated coordinates the
natural setting for the stability question \cite{HLZ,HLZmulti}.

\subsection{The integrated eigenvalue equations}
From this point on, $v,u,\eps$ denote integrated spectral perturbations,
not the gas fields in \eqref{eq:pde1}--\eqref{eq:pde3}. If
$\mathcal E$ is the integrated total-energy perturbation, then
\[
 \eps=\mathcal E-\hat u\,u=\mathcal E-(\hat v-1)u
\]
is the integrated thermal variable of \cite{HLZ}. Its derivative is related
to the linearized internal-energy perturbation $\delta e$ by
$\eps'=\delta e-\hat v'u$.
Linearizing the integrated equations about the integrated profile, and
using $\hat u=\hat v-1$ and $\hat u'=\hat v'$, one obtains for every shock in the
family the linearized integrated eigenvalue equations
\cite[(3.3)]{HLZ},
\begin{align*}
 &\lambda v+v'-u'=0,\\
 &\lambda u+u'+\frac\Gamma{\hat v}\eps'+\frac{\Gamma \hat v'}{\hat v}u
 +\frac{\hat v'-\Gamma \hat e}{\hat v^2}\,v'=\frac{u''}{\hat v},\\
 &\lambda\eps+\eps'+\left(\hat v'-\frac{\nu \hat v''}{\hat v}\right)u
 +\frac{\Gamma \hat e-(\nu+1)\hat v'}{\hat v}\,u'+\frac{\nu \hat e'}{\hat v^2}\,v'
 =\frac\nu{\hat v}\eps''.
\end{align*}
The coefficients involve only the profile and its derivatives. In the
strong-shock limit $e_-=0$, the profile equations \eqref{eq:profile} give
$\hat v'-\Gamma \hat e=-\hat va$ and $\Gamma \hat e-(\nu+1)\hat v'=\hat va-\nu \hat v'$, where $a=1-\hat v$.
The limiting integrated eigenvalue equations are therefore
\begin{align}
 \lambda v+v'&=u',\label{eq:ev1}\\
 \lambda u+u'+\frac\Gamma{\hat v}\eps'+\frac{\Gamma \hat v'}{\hat v} u
 -\frac a{\hat v}v'&=\frac{u''}{\hat v},\label{eq:ev2}\\
 \lambda\eps+\eps'+\left(\hat v'-\frac{\nu \hat v''}{\hat v}\right)u
 +\left(a-\frac{\nu \hat v'}{\hat v}\right)u'
 +\frac{\nu \hat e'}{\hat v^2}v'&=\frac\nu{\hat v}\eps''.\label{eq:ev3}
\end{align}
The first equation is the integrated conservation of mass. The second is
the integrated momentum equation, which is diffusive in $u$. The third is
the integrated thermal equation, which is diffusive in $\eps$ and is coupled
to the mass perturbation through the term $\nu \hat e'v'/\hat v^2$. This last
coupling is the one removed in Section~\ref{sec:thermal}.

\subsection{Behavior at spatial infinity}
Equations~\eqref{eq:ev1}--\eqref{eq:ev3} form a five-dimensional
first-order system $W'=A(x,\lambda)W$ in $W=(v,u,\eps,u',\eps')^T$, whose
coefficients converge exponentially to limits $A_\pm(\lambda)$ as
$x\to\pm\infty$ \cite[Lemma 2.3]{HLZ}. For physical parameters and
$\Rea\lambda>0$, the system has consistent splitting. The solutions
decaying at $-\infty$ form a space of dimension two, and the solutions
decaying at $+\infty$ form a space of dimension three
\cite[Lemma 3.3]{HLZ}. We call these the incoming and outgoing spaces.

The strong-shock limit changes the upstream behavior in an instructive way.
At $x=-\infty$ the limiting profile has $\hat v=1$ and $\hat e=0$, so the frozen
system is block triangular. It consists of the thermal
convection--diffusion equation $\lambda\eps+\eps'=\nu\eps''$ and a
pressureless block for $(v,u)$ forced by $\Gamma\eps'$.
The spatial exponents are
\[
 -\lambda,\qquad \frac{1\pm\sqrt{1+4\lambda}}{2},\qquad
 \frac{1\pm\sqrt{1+4\nu\lambda}}{2\nu}.
\]
The incoming space corresponds to the two exponents with positive real
part, which remain separated from the others on the closed right
half-plane. The transport exponent $-\lambda$ belongs to neither space.
It lies on the imaginary axis when $\lambda$ does, so consistent splitting
fails at $-\infty$ for imaginary $\lambda$ in the limit, reflecting the loss
of the genuine coupling condition when $e_-=0$ \cite[Remark 3.4]{HLZ}.
Nevertheless the incoming space extends analytically to the closed right
half-plane. At $x=+\infty$ the shock is noncharacteristic, and the outgoing
space likewise extends analytically to $\Rea\lambda\ge0$
\cite{MZ1}, \cite[Lemma 3.3]{HLZ}.

\subsection{The Evans function}
Choose bases $V^-(\lambda)$ of the unstable subspace of $A_-(\lambda)$ and
$V^+(\lambda)$ of the stable subspace of $A_+(\lambda)$, analytic in
$\lambda$ and continuous in the parameters, for example by Kato's ODE
\cite[Lemma 3.5]{HLZ}. The corresponding solutions
$W^-_1,W^-_2$ and $W^+_3,W^+_4,W^+_5$ of $W'=AW$, asymptotic to
$e^{A_-x}V^-$ as $x\to-\infty$ and to $e^{A_+x}V^+$ as $x\to+\infty$,
span the incoming and outgoing spaces. The Evans function is
\[
 D(\lambda)=\det\bigl(W^+_3,W^+_4,W^+_5,W^-_1,W^-_2\bigr)\big|_{x=0}.
\]
It vanishes exactly when the incoming and outgoing spaces intersect
nontrivially, and changing either basis multiplies it by a nonzero factor.
By \cite[Proposition 3.8]{HLZ}, $D$ is analytic in $\lambda$ on the closed
right half-plane and continuous in $(\Gamma,\nu,v_+)$ for $\Gamma>0$,
$\nu>0$, and $v_*\le v_+<1$, up to and including the strong-shock limit.
On $\{\Rea\lambda\ge0\}\setminus\{0\}$ its zeros correspond to admissible
solutions in the sense of Definition~\ref{def:admissible} below. Uniform bounds on the size of possible zeros are
proved in \cite[Proposition 4.1]{HLZ} and recalled in
Lemma~\ref{lem:highfrequency}.

At $\lambda=0$ the outgoing space requires care. It is defined by
continuation from $\Rea\lambda>0$, and its continued elements need not
decay. Slow modes whose exponents tend to zero with $\lambda$ approach
bounded, nonzero limits in a prescribed subspace at $+\infty$. Ruling out
solutions that decay at both ends is therefore not enough at zero. This
leads to the following notion.

\begin{definition}[Admissible spectral solution]\label{def:admissible}
In the strong-shock limit and for $\Rea\lambda\ge0$, an admissible solution
belongs to both the incoming and outgoing spaces continued from the
physical problem with $\Rea\lambda>0$. Continuation is understood in
both the shock parameter and the spectral parameter, as in
\cite[Lemmas 3.3 and 3.6]{HLZ}.
\end{definition}

For nonzero $\lambda$ in this half-plane, the continued incoming
solutions decay upstream and the outgoing solutions decay downstream.
The neutral upstream transport branch is not an incoming mode. At zero,
the incoming solutions still decay, whereas outgoing solutions may
approach nonzero constants in a prescribed downstream subspace.
Sections~\ref{sec:admissible} and~\ref{sec:zero} describe these behaviors
explicitly. Theorem~\ref{thm:main} is therefore
equivalent to showing that no nontrivial admissible solution exists at any
$\Rea\lambda\ge0$.

\section{Cancellation of the thermal mass coupling}\label{sec:thermal}
For a limiting profile with arbitrary positive $\Gamma,\nu$, define
\begin{equation}\label{eq:cd}
 d=\frac{a'}a,\qquad
 c=\frac{\nu \hat e'}{a\hat v}=\frac{\hat e-a^2/2}{a}.
\end{equation}
The profile equations imply
\begin{equation}\label{eq:cdident}
 d+\Gamma c=\hat v-\frac\Gamma2a=s,\qquad
 s'=-\frac{ad}{a_*},
\end{equation}
and hence
\begin{equation}\label{eq:cderiv}
 c'=\frac{\hat v}\nu c-d(a+c),\qquad
 d'=-\frac{\Gamma+2}{2}ad-\Gamma c'.
\end{equation}
The identity $d=\hat v-\Gamma \hat e/a$ gives \eqref{eq:cdident}.
Differentiating $c=\hat e/a-a/2$ and $d+\Gamma c=s$ yields
\eqref{eq:cderiv}.

Set
\begin{equation}\label{eq:h}
 h=\eps+c u.
\end{equation}
The map $(u,\eps)\mapsto(u,h)$ is triangular with determinant one.
It retains an independent thermal unknown and imposes no relation between
thermal and velocity perturbations. Adding $c$ times
\eqref{eq:ev2} to \eqref{eq:ev3} cancels $v'$ exactly, since
\[
 \frac{\nu \hat e'}{\hat v^2}-\frac{ca}{\hat v}=0.
\]
Using $\eps'=h'-cu'-c'u$ and
$\eps''=h''-cu''-2c'u'-c''u$, the transformed residuals are
\begin{align}
 R_2={}&\lambda u+\frac{\hat v-\Gamma c}{\hat v}u'
 +\frac\Gamma{\hat v}h'+\frac{\Gamma(\hat v'-c')}{\hat v}u
 -\frac a{\hat v}v'-\frac{u''}{\hat v}=0,\label{eq:r2}\\
 R_h={}&\lambda h+\frac{\hat v+\Gamma c}{\hat v}h'
 +\frac{\hat va-\nu \hat v'+2\nu c'-\Gamma c^2}{\hat v}u'\notag\\
 &+\frac{\hat v\hat v'-\nu \hat v''+\Gamma c\hat v'-(\hat v+\Gamma c)c'+\nu c''}{\hat v}u
 -\frac\nu{\hat v}h''-\frac{(1-\nu)c}{\hat v}u''=0.\label{eq:rh}
\end{align}
The mass derivative now appears only in the momentum equation. This
cancellation is valid for every positive $\Gamma,\nu$. The parameter
restrictions enter only in the coercivity estimate of
Section~\ref{sec:coercivity}.

\section{The energy identity}\label{sec:energy}
The estimate uses one multiplier for each equation. The mass and momentum
multipliers are fixed, and the thermal multiplier contains a positive
weight $\omega$ along the profile, chosen in Section~\ref{sec:coercivity}.
The constant choice $\omega=\tht$, where
\begin{equation}\label{eq:theta}
 \tht=\frac{a_*v_*}{\Gamma}=\frac{2}{(\Gamma+2)^2}=\frac{a_*^2}2,
\end{equation}
is used for $1\le\nu\le\Gamma+1$, and a variable weight is used otherwise. Both choices tend to $\tht$ downstream.

\subsection{The thermal equation in terms of the translation factor}
Write the velocity perturbation in units of the translation mode $u=-a$,
\[
 \phi=\frac ua,\qquad \Phi=\sqrt a\,\phi'=\frac{u'-du}{\sqrt a}.
\]
Since $(\partial_x-d)a=0$, the momentum operator factors through
$\partial_x-d$. Indeed, $\hat v\,R_2$ in \eqref{eq:r2} equals
\begin{equation}\label{eq:momfactor}
 \lambda\hat vu-\Bigl(\partial_x+\Bigl(1-\frac\Gamma2\Bigr)a\Bigr)
 (\partial_x-d)u+au'+\Gamma h'-av',
\end{equation}
as one checks with $\hat v'=-ad$ and \eqref{eq:cderiv}. Put
\begin{equation}\label{eq:F}
 b_h=\hat v+\Gamma c,\qquad
 b_u=\hat va+\nu ad+2\nu c'-\Gamma c^2,\qquad
 F=b_u+(\nu-1)(2cd-c'),
\end{equation}
so that $b_h/\hat v$ and $b_u/\hat v$ are the coefficients of $h'$ and $u'$
in \eqref{eq:rh}. In terms of $\sqrt a\,\Phi=u'-du$, the thermal residual
becomes
\begin{equation}\label{eq:rhU}
 \hat vR_h=\lambda\hat vh+b_hh'-\nu h''
 +(\nu-1)c\bigl(\sqrt a\,\Phi\bigr)'
 +\bigl(b_u+(\nu-1)cd\bigr)\sqrt a\,\Phi .
\end{equation}
The zeroth-order coefficient of $u$ has disappeared. As in
Remark~\ref{rem:translation} below, this happens because the translation
mode $u=-a$, $h=0$ solves the thermal equation at $\lambda=0$, and
substitution of \eqref{eq:cderiv} confirms it directly.

\subsection{The energy balance}
Let $\omega>0$ be a smooth function along the profile with bounded
logarithmic derivative. Multiply the mass equation by $\bar v$,
\eqref{eq:r2} by $\hat v\bar u/a$, and \eqref{eq:rh} by
$\hat v\bar h/(a\omega)$. Take real parts and integrate. The two mass
couplings have the simple form
\[
 -\Rea(u'\bar v+v'\bar u)=-\bigl(\Rea(u\bar v)\bigr)',
\]
so no cross term depending on the spectral parameter remains. With
\[
 E_\omega=\int_{\R}\Bigl(|v|^2+\frac{\hat v}a|u|^2
 +\frac{\hat v}{a\omega}|h|^2\Bigr)\dd
\]
one obtains, provided the boundary terms vanish,
\begin{equation}\label{eq:energy}
 0=(\Rea\lambda)E_\omega+Q_\omega .
\end{equation}
Appendix~\ref{app:identity} gives $Q_\omega$ as the integral of an explicit
quadratic form in $(\phi,\Phi,h,h')$, together with the pointwise divergence
identity behind \eqref{eq:energy}. Section~\ref{sec:admissible} justifies
the integrations for admissible solutions with $\lambda\ne0$.

For the coercivity estimate we write $Q_\omega$ in a normalized thermal
variable. Put
\begin{equation}\label{eq:xi}
 \xi=\frac{h}{\sqrt{a\omega}},\qquad
 \delta=d+\frac{\omega'}{\omega},\qquad
 \beta=\frac{(\nu-1)c}{2\sqrt\omega},\qquad
 \kappa=\frac{\beta^2}{\nu}=\frac{(\nu-1)^2c^2}{4\nu\omega},
\end{equation}
so that $\delta$ is the logarithmic rate at which $a\omega$ grows. Define the
thermal potential and the coupling coefficient
\begin{align}
 P&=\frac{ad}2+\frac{\hat v\delta}2-\frac{(\nu\delta-\Gamma c)^2}{4\nu},
 \label{eq:potential}\\
 m&=\frac1{2\sqrt\omega}\Bigl[F+(\nu-1)c\,\frac{\omega'}\omega-\Gamma\omega
 -(\nu-1)c\Bigl(\frac\delta2-\frac{\Gamma c}{2\nu}\Bigr)\Bigr].
 \label{eq:coupling}
\end{align}
The coupling has a simpler form. Let
\begin{equation}\label{eq:Tn}
 \mathcal T=\frac{a^2}{2}+\frac{3\nu+1}{2\nu}ac+\frac{\nu+1}{2\nu}c^2
 +\frac{(\nu-1)(\nu-2)}{2\nu\Gamma}cd .
\end{equation}
Substituting $\delta=d+\omega'/\omega$ in \eqref{eq:coupling} isolates the
term $(\nu-1)c\,\omega'/(2\omega)$. In the rest, the definition
\eqref{eq:F} of $F$ and the profile equation \eqref{eq:cderiv} for $c'$ give
\[
 F-(\nu-1)c\Bigl(\frac d2-\frac{\Gamma c}{2\nu}\Bigr)
 =\hat v\Bigl(a+\frac{\nu+1}\nu c\Bigr)-d\Bigl(a+\frac{5-\nu}2c\Bigr)
 -\frac{\Gamma(\nu+1)}{2\nu}c^2 ,
\]
and the normalization $\hat v=d+\Gamma(a/2+c)$ from \eqref{eq:cdident} turns
the right side into $\Gamma\mathcal T$. Therefore
\begin{equation}\label{eq:coupling-T}
 m=\frac1{2\sqrt\omega}\Bigl[\Gamma(\mathcal T-\omega)
 +\frac{(\nu-1)c}2\,\frac{\omega'}\omega\Bigr].
\end{equation}
For a constant weight the coupling is proportional to $\mathcal T-\omega$, and
a weight that varies along the profile adds the second term.

\begin{proposition}[Normalized energy form]\label{prop:energyform}
If the boundary terms vanish, then
\begin{equation}\label{eq:Qomega}
 Q_\omega=\int_\R\Bigl[\Bigl(1-\frac\Gamma2\Bigr)a\,a'|\phi|^2
 +\nu\Bigl|\xi'+\frac{\Gamma c}{2\nu}\xi-\frac\beta\nu\Phi\Bigr|^2
 +(1-\kappa)|\Phi|^2+2m\Rea(\Phi\bar\xi)+P|\xi|^2\Bigr]\dd .
\end{equation}
\end{proposition}
\begin{proof}
Start from the form \eqref{eq:rawform} of Appendix~\ref{app:identity}. Its
momentum part is $(1-\Gamma/2)a\,a'|\phi|^2+|\Phi|^2$. Indeed, with $u=a\phi$
the factorization \eqref{eq:momfactor} gives
$-\bigl((a\phi')'+(1-\Gamma/2)a^2\phi'\bigr)\bar\phi+a(a\phi)'\bar\phi$. The
first term integrates to $a|\phi'|^2=|\Phi|^2$. The real part of the second
is $-(1-\Gamma/2)a^2(|\phi|^2)'/2$, which integrates by parts to
$(1-\Gamma/2)a\,a'|\phi|^2$. The real part of the last is the derivative
$(a^2|\phi|^2)'/2$.

In the thermal part substitute $h=\sqrt{a\omega}\,\xi$, so that
$h'=\sqrt{a\omega}\,(\xi'+\delta\xi/2)$. The diffusion term becomes
$\nu|\xi'+\delta\xi/2|^2$, and after the cross term $\Rea(\xi'\bar\xi)$ is
integrated by parts the coefficient of $|\xi|^2$ is
\[
 -\frac{\nu\delta^2}4-\frac{b_h'}2+\frac{b_h\delta}2 .
\]
This potential contains $c'$. Since
\[
 \nu|\xi'|^2=\nu\Bigl|\xi'+\frac{\Gamma c}{2\nu}\xi\Bigr|^2
 -\Bigl(\frac{\Gamma c}2|\xi|^2\Bigr)'+\frac{(\Gamma c)'}2|\xi|^2
 -\frac{(\Gamma c)^2}{4\nu}|\xi|^2,
\]
shifting the first-order thermal term by $\Gamma c\,\xi/(2\nu)$ removes
$c'$ from the potential, and with $\hat v'=-ad$ the potential becomes $P$.
Finally, the coupling between $\Phi$ and $\xi'$ in \eqref{eq:rawform} is
$-2\beta\Rea(\Phi\bar\xi')$. Completing the square in $\xi'$ absorbs it,
leaves $(1-\kappa)|\Phi|^2$, and collects the remaining coupling between
$\Phi$ and $\xi$ into $2m\Rea(\Phi\bar\xi)$. Apart from the flux of
Appendix~\ref{app:identity}, the only boundary term of these steps is
$(\Gamma c-\nu\delta)|\xi|^2/2$. Lemma~\ref{lem:boundary} shows that all of
them vanish at both ends.
\end{proof}

The momentum dissipation $(1-\Gamma/2)a\,a'$ is $(1-\Gamma/2)/2$ times the
slope of $\hat\sigma^2$, where $\hat\sigma=a$ is the total stress
(Remark~\ref{rem:stress}), and it is positive exactly when $\gamma<3$. Coercivity therefore reduces to positivity of the
two-by-two matrix
\begin{equation}\label{eq:matrixM}
 \begin{pmatrix}1-\kappa&m\\ m&P\end{pmatrix}
\end{equation}
along the profile.

The coefficients have direct meanings, recorded in the following remarks,
which are not used in the proofs.

\begin{remark}[The thermal variable removes translation]\label{rem:translation}
By Proposition~\ref{prop:zero} and its proof, translation of the profile
gives the integrated solution $u=v=-a$, $\eps=\hat e-a^2/2$ at
$\lambda=0$. Since $c=(\hat e-a^2/2)/a$, the variable $h=\eps+cu$
vanishes identically on it. Thus $h$ measures the thermal perturbation
relative to the family of translated profiles, which is why the change of
variables removes the coupling to the mass derivative, and why the
zeroth-order coefficient of $u$ disappears from \eqref{eq:rhU}.
\end{remark}

\begin{remark}[The multipliers symmetrize the total stress]\label{rem:stress}
Let $\sigma=p-u_x/v$ be the total normal stress in \eqref{eq:pde2}.
Integrating the momentum equation of the limiting profile gives
$\hat\sigma=a$, and linearizing $\sigma$ in $v$ gives
$\partial_v\sigma=-\hat\sigma/\hat v=-a/\hat v$, the coefficient of $v'$
in \eqref{eq:ev2}. The mass and momentum equations thus form a $p$-system
with stress $\sigma$, and the multipliers $\bar v$ and $\hat v\bar u/a$ are
its symmetrizer. In particular $a$ is at the same time the translation
mode of the velocity and the total stress of the profile.
\end{remark}

\begin{remark}[The downstream value of the weight]
Both weights used in Section~\ref{sec:coercivity} tend to $\tht$ as
$x\to+\infty$. By the Rankine--Hugoniot relations \eqref{eq:endstates},
$\tht=a_*^2/2=e_+$ is the downstream internal energy. Heuristically, the
thermal multiplier is then the entropy symmetrizer frozen at the downstream
temperature.
\end{remark}

\section{Coercivity}\label{sec:coercivity}
In this section we choose the weight $\omega$ and prove that the matrix
\eqref{eq:matrixM} is positive definite along the profile for every
$(\Gamma,\nu)\in\mathcal R$. For $1\le\nu\le\Gamma+1$ we take the constant
weight $\omega=\tht$. Otherwise, that is above the reference line or below
$\nu=1$, we take a weight that follows the profile. No continuity of the
weight is needed, since each parameter pair has its own energy estimate.

\subsection{Below the reference line}
Assume $0<\Gamma<2$ and $1\le\nu\le\Gamma+1$, and take $\omega=\tht$. Then
$\delta=d$, and \eqref{eq:potential} and \eqref{eq:coupling-T} reduce to
\begin{equation}\label{eq:lower-coeffs}
 P=\frac d2-\frac{(\nu d-\Gamma c)^2}{4\nu},\qquad
 m=\frac{\Gamma}{2\sqrt\tht}\bigl(\mathcal T-\tht\bigr),
\end{equation}
with $\mathcal T$ as in \eqref{eq:Tn}. By Lemma~\ref{lem:slopestrip}, the slope satisfies $1/(\Gamma+1)\le z\le1/\nu$,
and since $\Gamma c=s-d=d(1-z)/z$ this is the cone
\begin{equation}\label{eq:lower-cone}
 (\nu-1)d\le\Gamma c\le\Gamma d .
\end{equation}
In particular $c\le d<1$, and $(\Gamma+1)c\le d+\Gamma c=s<1$.

\begin{lemma}[The quadratic $\mathcal T$]\label{lem:thermal-bound}
Under \eqref{eq:lower-cone}, with $0<\Gamma<2$ and $1\le\nu\le\Gamma+1$,
\begin{equation}\label{eq:Tbounds}
 \tht(1-d)^2\le\mathcal T\le\tht\bigl[2-(1-d)^2\bigr],
 \qquad\text{equivalently}\qquad
 \Bigl|\frac{\mathcal T}\tht-1\Bigr|\le d(2-d).
\end{equation}
\end{lemma}
\begin{proof}
Set
\[
 \alpha_1=\frac{3\nu+1}{2\nu}-\frac{2\Gamma}{\Gamma+2},\qquad
 \alpha_2=\frac{\nu+1}{2\nu}-\tht\Gamma^2,\qquad
 \alpha_3=\frac{(\nu-1)(\nu-2)}{2\nu\Gamma}.
\]
Since $2\Gamma/(\Gamma+2)<1$ and $\tht\Gamma^2=2(\Gamma/(\Gamma+2))^2<1/2$, both
$\alpha_1$ and $\alpha_2$ are positive. The normalization
\eqref{eq:cdident} gives $1-d=a/a_*+\Gamma c$, and expansion yields
\[
 \mathcal T-\tht(1-d)^2=\alpha_1ac+\alpha_2c^2+\alpha_3cd .
\]
For $\nu\ge2$ every term is nonnegative. For $1\le\nu\le2$ the same
expression equals
\[
 \alpha_1ac+\Bigl(\alpha_2-\frac{2-\nu}{2\nu}\Bigr)c^2
 +\frac{2-\nu}{2\nu\Gamma}c\bigl(\Gamma c-(\nu-1)d\bigr),
\]
and every term is nonnegative, because
$\alpha_2-(2-\nu)/(2\nu)=1-1/(2\nu)-\tht\Gamma^2>0$ and \eqref{eq:lower-cone}
controls the last term. This proves the lower bound.

For the upper bound put $\alpha_4=(3\nu-1)/(2\nu)-2\tht-\alpha_3$. Expansion gives
\[
 \tht\bigl[2-(1-d)^2\bigr]-\mathcal T
 =\frac{\nu-1}{2\nu}ad+\alpha_4cd+(d-c)\bigl(\alpha_1a+\alpha_2c+2\tht d\bigr).
\]
Here $d-c\ge0$. If $1\le\nu\le2$, then $\alpha_3\le0$ and $\alpha_4\ge1-2\tht>0$.
If $\nu\ge2$, the condition $\nu-1\le\Gamma$ gives $\alpha_3\le(\nu-2)/(2\nu)$,
and then $\alpha_4\ge1+1/(2\nu)-2\tht>0$. Every term is nonnegative.
\end{proof}

The cone \eqref{eq:lower-cone} gives $(\nu-\Gamma)d\le\nu d-\Gamma c\le d$,
and $\nu-\Gamma\ge1-\Gamma>-1$, so $|\nu d-\Gamma c|\le d$. Since $\nu\ge1$,
\begin{equation}\label{eq:P-lower}
 P\ge\frac d2-\frac{d^2}4=\frac{d(2-d)}4>0 .
\end{equation}
With $d(2-d)\in(0,1)$, Lemma~\ref{lem:thermal-bound} gives
\begin{equation}\label{eq:m-bound}
 m^2=\frac{\Gamma^2\tht}{4}\Bigl(\frac{\mathcal T}\tht-1\Bigr)^2
 \le\frac{\Gamma^2\tht}{4}\,d(2-d)\le\Gamma^2\tht P .
\end{equation}
The remaining bounds use only the parameters. Since $c<1/(\Gamma+1)$ and
$(\nu-1)^2/\nu$ increases with $\nu$,
\begin{equation}\label{eq:kappa-bound}
 \kappa=\frac{(\nu-1)^2c^2}{4\nu\tht}
 \le\frac{\Gamma^2}{4\tht(\Gamma+1)^3}
 =\frac{\bigl((\Gamma+1)^2-1\bigr)^2}{8(\Gamma+1)^3}
 <\frac{\Gamma+1}8<\frac38 ,
\end{equation}
and $\Gamma^2\tht<1/2$. Therefore
\[
 (1-\kappa)P-m^2\ge(1-\kappa-\Gamma^2\tht)P>\frac P8 .
\]

\begin{proposition}[Coercivity with the constant weight]\label{prop:lower}
Let $0<\Gamma<2$ and $1\le\nu\le\Gamma+1$, and take $\omega=\tht$. If the
boundary terms vanish, then
\[
 Q_\tht\ge\int_\R\Bigl[\Bigl(1-\frac\Gamma2\Bigr)a\,a'|\phi|^2
 +\nu\Bigl|\xi'+\frac{\Gamma c}{2\nu}\xi-\frac\beta\nu\Phi\Bigr|^2
 +(1-\kappa)\Bigl|\Phi+\frac m{1-\kappa}\xi\Bigr|^2
 +\frac P8|\xi|^2\Bigr]\dd .
\]
\end{proposition}
\begin{proof}
Complete the square in \eqref{eq:Qomega}. The last coefficient is
$P-m^2/(1-\kappa)\ge(1-\kappa-\Gamma^2\tht)P/(1-\kappa)\ge P/8$.
\end{proof}

\subsection{The variable weight}
Assume $(\Gamma,\nu)\in\mathcal R$ with either $\nu>\Gamma+1$ or $\nu<1$.
We use the slope ratio
\[
 \rho=\frac{1-z}{z}=\frac{\Gamma c}{d},
\]
and the corresponding downstream value $\rho_r=(1-r)/r$, where $r$ is the
downstream slope of Lemma~\ref{lem:slopestrip}. Substituting $z=1/(1+\rho)$
in $f_0(r)=0$ shows that $\rho_r$ is the positive root of
\begin{equation}\label{eq:slope-root}
 \rho^2+\Bigl(\frac{2\nu}\Gamma-(\nu-1)\Bigr)\rho-2\nu=0 .
\end{equation}
By Lemma~\ref{lem:slopestrip}, along the profile
\begin{equation}\label{eq:sharp-slope}
 \Gamma<\rho_r<\rho<\nu-1\quad(\nu>\Gamma+1),\qquad
 0<\rho<\rho_r\quad(\nu<1).
\end{equation}
In both cases $\rho>0$.

Take
\begin{equation}\label{eq:upper-weight}
 \omega=\frac{\varpi^2}{2},\qquad \varpi=a+\frac43c,
 \qquad \vartheta=\frac a\varpi\in(0,1).
\end{equation}
Since $c\to0$ downstream, $\omega\to a_*^2/2=\tht$ there. By
\eqref{eq:cderiv} and the normalization $\hat v=d(1+\rho)+\Gamma a/2$,
\[
 \frac{c'}d=\frac{1+\rho-\nu}{\nu}\,c+\frac{\rho-2\nu}{2\nu}\,a,\qquad
 \frac{\varpi'}{d\varpi}=(1-\vartheta)\frac{1+\rho-\nu}{\nu}
 +\vartheta\,\frac{2\rho-\nu}{3\nu},
\]
using $a=\vartheta\varpi$ and $4c/3=(1-\vartheta)\varpi$. Since
$\delta=d+2\varpi'/\varpi$, the logarithmic rate is an affine interpolation
in $\vartheta$,
\begin{equation}\label{eq:upsilon}
 \frac\delta d=\upsilon=(1-\vartheta)\upsilon_0+\vartheta\upsilon_1,
 \qquad \upsilon_0=\frac{2+2\rho-\nu}{\nu},\quad
 \upsilon_1=\frac{\nu+4\rho}{3\nu}.
\end{equation}
Define
\begin{align*}
 \kappa_0&=\frac{9(\nu-1)^2}{32\nu},&
 m_0&=\frac{(11\nu-9)\rho-9\nu(\nu-1)}{24\sqrt2\,\nu},&
 p_0&=\frac{3\rho^2+4\rho-\nu(\nu-2)}{4\nu},\\
 m_1&=\frac{(\nu+3)\rho}{6\sqrt2\,\nu},&
 p_1&=\frac{23\rho^2+(4\nu+24)\rho+6\nu-\nu^2}{36\nu},
\end{align*}
and form
\[
 \mathsf E_0=\begin{pmatrix}1-\kappa_0&m_0\\ m_0&p_0\end{pmatrix},\qquad
 \mathsf E_1=\begin{pmatrix}1&m_1\\ m_1&p_1\end{pmatrix},\qquad
 \mathsf e=\begin{pmatrix}\dfrac{3(\nu-1)}{4\sqrt{2\nu}}\\[6pt]
 \dfrac{\sqrt\nu\,(\upsilon_1-\upsilon_0)}2\end{pmatrix}.
\]
These depend on the profile only through $\rho$. The labels $0$ and $1$
refer to $\vartheta=0$ and $\vartheta=1$ with $\rho$ held fixed; they are not
the two endstates of the shock.

\begin{proposition}[Interpolation of the energy matrix]\label{prop:convexity}
With the weight \eqref{eq:upper-weight}, the matrix \eqref{eq:matrixM},
written in the variables $(\Phi,d\,\xi)$, satisfies
\begin{equation}\label{eq:matrix-decomposition}
 \begin{pmatrix}1-\kappa&m/d\\ m/d&P/d^2\end{pmatrix}
 =(1-\vartheta)\mathsf E_0+\vartheta\mathsf E_1
 +\vartheta(1-\vartheta)\,\mathsf e\mathsf e^T
 +\begin{pmatrix}0&0\\0&\dfrac a{4d}(2+\Gamma\upsilon)\end{pmatrix}.
\end{equation}
\end{proposition}
\begin{proof}
Since $c/\varpi=3(1-\vartheta)/4$, one has $\kappa=\kappa_0(1-\vartheta)^2$,
which is the first diagonal entry. The coupling is
\[
 \frac md=(1-\vartheta)m_0+\vartheta m_1
 +\frac{3(\nu-1)(\upsilon_1-\upsilon_0)}{8\sqrt2}\,\vartheta(1-\vartheta).
\]
The quadratic remainder comes from the second term in \eqref{eq:coupling-T}.
Since $\sqrt\omega=\varpi/\sqrt2$, $c/\varpi=3(1-\vartheta)/4$, and
$2\varpi'/(d\varpi)=\upsilon-1$, that term divided by $d$ equals
$3(\nu-1)(1-\vartheta)(\upsilon-1)/(8\sqrt2)$, and
$(1-\vartheta)(\upsilon-1)=(1-\vartheta)(\upsilon_0-1)
+\vartheta(1-\vartheta)(\upsilon_1-\upsilon_0)$. The first term of
\eqref{eq:coupling-T} divided by $d$ is affine in $\vartheta$ by direct
substitution of $a=\vartheta\varpi$ and $c=3(1-\vartheta)\varpi/4$.
For the second diagonal entry, the normalization \eqref{eq:cdident} and
\eqref{eq:potential} give
\[
 \frac P{d^2}=\frac a{4d}(2+\Gamma\upsilon)
 +\frac{(1+\rho)\upsilon}{2}-\frac{(\nu\upsilon-\rho)^2}{4\nu}.
\]
The last two terms form a concave quadratic in $\upsilon$. On the affine
interpolation \eqref{eq:upsilon} it equals
\[
 (1-\vartheta)p_0+\vartheta p_1+\frac{\nu(\upsilon_1-\upsilon_0)^2}{4}\,
 \vartheta(1-\vartheta).
\]
The three interpolation remainders are exactly the entries of
$\vartheta(1-\vartheta)\mathsf e\mathsf e^T$.
\end{proof}

The interior of the profile thus contributes a positive square, and it
suffices to prove that $\mathsf E_0$ and $\mathsf E_1$ are positive definite and
that $\upsilon>0$. The identity is algebraic and uses only $\rho>0$, not the
side of the reference line.

\subsubsection*{The matrix $\mathsf E_1$}
On $\mathcal R$ one has $2/3<\nu<5$. Hence
\begin{equation}\label{eq:det1}
 \det\mathsf E_1=
 \frac{2\nu^2(6-\nu)+8\nu(\nu+6)\rho+(40\nu-\nu^2-9)\rho^2}{72\nu^2}>0,
\end{equation}
since all three coefficients in the numerator are positive. The same range
controls the first diagonal entry of $\mathsf E_0$, because
\begin{equation}\label{eq:kappa0}
 \frac9{10}-\kappa_0=\frac{9(5\nu-1)(5-\nu)}{160\nu}>0 .
\end{equation}

\subsubsection*{The determinant of $\mathsf E_0$}
Since $1-\kappa_0=(32\nu-9(\nu-1)^2)/(32\nu)$, writing
$(1-\kappa_0)p_0-m_0^2$ over the common denominator $1152\nu^2$ gives
$\det\mathsf E_0=\Pi_\nu(\rho)/(1152\nu^2)$, with $\Pi_\nu$ the polynomial
\eqref{eq:Pi}. Since
$9(32\nu-9(\nu-1)^2)=288\nu(1-\kappa_0)$, it can be written as
\[
 \Pi_\nu(\rho)=288\nu(1-\kappa_0)\bigl(3\rho^2+4\rho-\nu(\nu-2)\bigr)
 -\bigl((11\nu-9)\rho-9\nu(\nu-1)\bigr)^2 ,
\]
which is $1152\nu^2$ times $(1-\kappa_0)p_0-m_0^2$. Two values of $\rho$ are
distinguished. At $\rho=0$,
\begin{equation}\label{eq:Pi0}
 \Pi_\nu(0)=-18\nu(5\nu-9)(5\nu-1),
\end{equation}
which is nonnegative exactly for $1/5\le\nu\le9/5$. At $\rho=\nu-1$, the
upstream value of the slope ratio, one has $p_0=(2\nu^2-1)/(4\nu)$ and
$m_0=(2\nu-9)(\nu-1)/(24\sqrt2\,\nu)$, and with
$(\nu-1)^2=32\nu\kappa_0/9$,
\begin{equation}\label{eq:Pit}
 \Pi_\nu(\nu-1)=\frac{32\nu}9\Bigl[81(1-\kappa_0)(2\nu^2-1)
 -\kappa_0(2\nu-9)^2\Bigr].
\end{equation}
The decisive fact is that the determinant increases with the slope ratio.

\begin{lemma}[Monotonicity of the endpoint determinant]\label{lem:det-monotone}
For $1<\nu<5$, $\Pi_\nu$ is strictly increasing on $0\le\rho\le\nu-1$, and
$\Pi_\nu(\nu-1)>0$. For $2/3<\nu\le9/5$, $\Pi_\nu$ is strictly increasing on
$[0,\infty)$. Consequently, for $9/5<\nu<5$ the polynomial $\Pi_\nu$ has
exactly one zero $\rho_E$ in $(0,\nu-1)$, and for $0\le\rho\le\nu-1$ one has
$\Pi_\nu(\rho)\ge0$ exactly when $\rho\ge\rho_E$.
\end{lemma}
\begin{proof}
Put $t=\nu-1$. The derivative $\Pi_\nu'$ is affine in $\rho$, with
\[
 \Pi_\nu'(0)=1152\nu(1-\kappa_0)+18\nu t(11\nu-9),\qquad
 \Pi_\nu'(t)=288\nu(1-\kappa_0)(6t+4)+2t(11\nu-9)(9-2\nu).
\]
Let $1<\nu<5$. By \eqref{eq:kappa0}, $1-\kappa_0>1/10$, so the first terms are
positive. The second term of $\Pi_\nu'(0)$ is positive, and so is the second
term of $\Pi_\nu'(t)$ when $\nu\le9/2$. If $9/2<\nu<5$, that term is at least
$-2\cdot4\cdot46=-368$, while the first term exceeds
$288\cdot\frac92\cdot\frac1{10}\cdot25=3240$. Hence $\Pi_\nu'>0$ on $[0,t]$.

If $\nu\le9/5$, then $\Pi_\nu(t)\ge\Pi_\nu(0)\ge0$ by \eqref{eq:Pi0}, with
strict inequality since $\Pi_\nu$ increases. If $9/5<\nu<5$, then
$\kappa_0<9/10$ gives $1-\kappa_0>\kappa_0/9$, and the bracket in
\eqref{eq:Pit} exceeds
\[
 \kappa_0\bigl[9(2\nu^2-1)-(2\nu-9)^2\bigr]=2\kappa_0(7\nu^2+18\nu-45),
\]
which is positive for $\nu\ge9/5$.

Now let $2/3<\nu\le9/5$, so that $-1/3<t\le4/5$. Then $\kappa_0\le1/10$,
since $\kappa_0$ is convex in $t$ and equals $3/64$ and $1/10$ at the two
ends. For $t<0$ the second term of $\Pi_\nu'(0)$ is at least
$-18\nu\cdot\frac13\cdot2=-12\nu$, so $\Pi_\nu'(0)\ge1152\nu\cdot\frac9{10}-12\nu>0$,
and for $t\ge0$ we showed this already. The leading coefficient of $\Pi_\nu$ is
\[
 A_\nu=864\nu(1-\kappa_0)-(11\nu-9)^2\ge\frac{3888}{5}\cdot\frac23
 -\Bigl(\frac{54}5\Bigr)^2>0,
\]
so $\Pi_\nu$ is a convex quadratic with positive slope at $0$, and it
increases on $[0,\infty)$.

Finally, for $9/5<\nu<5$ one has $\Pi_\nu(0)<0<\Pi_\nu(t)$ by \eqref{eq:Pi0}
and the second paragraph, and the last assertion follows from the
monotonicity on $[0,t]$.
\end{proof}

Since $\rho$ lies between $\rho_r$ and $\nu-1$ above the reference line, the
lemma shows that the single endpoint condition
\begin{equation}\label{eq:endpoint-criterion}
 \Pi_\nu(\rho_r)\ge0
\end{equation}
is sufficient for the positivity of $\mathsf E_0$ there.

If $\nu\le9/5$, then $\Pi_\nu(\rho)>0$ for every $\rho>0$ by \eqref{eq:Pi0} and
Lemma~\ref{lem:det-monotone}. This covers every case with $\nu<1$ and every
case with $\Gamma+1<\nu\le9/5$.

If $9/5<\nu<5$, the lemma shows that \eqref{eq:endpoint-criterion} holds
exactly when $\rho_r\ge\rho_E$. Solving \eqref{eq:slope-root} for $\Gamma$
gives
\[
 \Gamma=\frac{2\nu\rho_r}{2\nu+(\nu-1)\rho_r-\rho_r^2}.
\]
On $0<\rho<\nu-1$ the denominator is positive, and the right side increases
with $\rho_r$, since its derivative has the sign of $2\nu+\rho_r^2$. Hence
$\rho_r\ge\rho_E$ exactly when $\Gamma\ge\Gamma_E(\nu)$, which is the
condition defining $\mathcal R$ in \eqref{eq:region}. The same map sends
$\rho=\nu-1$ to $\Gamma=\nu-1$, so $\rho_E<\nu-1$ gives
$\Gamma_E(\nu)<\nu-1$. Along the profile $\rho>\rho_r\ge\rho_E$, so
$\Pi_\nu(\rho)>0$.

In every case, together with $\kappa_0<1$, this proves $\mathsf E_0>0$.

\subsubsection*{Positivity of $\upsilon$}
Positivity of $\mathsf E_0$ implies $p_0>0$. The identity
\[
 p_0=\frac{(1+\rho)\upsilon_0}{2}-\frac{(\nu\upsilon_0-\rho)^2}{4\nu}
\]
then forces $\upsilon_0>0$. Since $\upsilon_1>0$ as well, the interpolation
$\upsilon$ is positive, and the diagonal remainder in
\eqref{eq:matrix-decomposition} is positive.

\begin{proposition}[Coercivity with the variable weight]\label{prop:upper}
Let $(\Gamma,\nu)\in\mathcal R$ with $\nu>\Gamma+1$ or $\nu<1$, and take the
weight \eqref{eq:upper-weight}. Then the matrix \eqref{eq:matrixM} is positive
definite at every finite point of the profile.
\end{proposition}
\begin{proof}
Both endpoint matrices are positive definite, and the remaining terms in
\eqref{eq:matrix-decomposition} are nonnegative, so the left side of
\eqref{eq:matrix-decomposition} is positive definite for $0<\vartheta<1$.
\end{proof}

\subsection{Conclusion of the estimate}
\begin{proposition}[Coercivity]\label{prop:coercive}
Let $(\Gamma,\nu)\in\mathcal R$, and let $\omega$ be $\tht$ if
$1\le\nu\le\Gamma+1$ and \eqref{eq:upper-weight} otherwise. If the boundary
terms in the energy identity vanish, then $Q_\omega\ge0$, and $Q_\omega=0$
forces $u=h=0$.
\end{proposition}
\begin{proof}
By Propositions~\ref{prop:energyform}, \ref{prop:lower}, and~\ref{prop:upper},
the integrand of $Q_\omega$ is a sum of nonnegative terms, among them
$(1-\Gamma/2)a\,a'|\phi|^2$ with $a\,a'=a^2d>0$ and a positive definite form
in $(\Phi,\xi)$. If $Q_\omega=0$, then $\phi=0$ and $\xi=0$. Since $a$ and
$\omega$ are positive at every finite point, $u=h=0$.
\end{proof}

\begin{remark}[The constant weight and the diffusion curve]\label{rem:sharp}
For the constant weight $\omega=\tht$, the coefficient $1-\kappa$ of
$|\Phi|^2$ tends at the cold upstream state to $1-\kappa_{\rm up}$, where
\[
 \kappa_{\rm up}=\frac{(\nu-1)^4(\Gamma+2)^2}{8\Gamma^2\nu^3}.
\]
If $\kappa_{\rm up}>1$, a rapidly oscillating $\Phi$ supported near $-\infty$,
with $\xi'$ chosen to cancel the square in \eqref{eq:Qomega}, makes
$Q_\tht$ negative, since $(1-\kappa)|\Phi|^2$ is of second order in the
frequency and every remaining term is of lower order. So $\kappa_{\rm up}\le1$
is necessary for the estimate with the constant weight, which is therefore
limited to the set
\[
 \mathcal R_D=\bigl\{(\Gamma,\nu):\ 0<\Gamma<2,\ \nu\ge1,\
 (\nu-1)^4(\Gamma+2)^2\le8\Gamma^2\nu^3\bigr\},
\]
whose upper boundary $\nu=\nu_D(\Gamma)$, computed numerically, is the dashed curve in
Figure~\ref{fig:regions}. The variable weight \eqref{eq:upper-weight} has
$\kappa\le\kappa_0<1$, and this obstruction does not apply to it.
Numerical evaluation shows that $\mathcal R_D$ lies inside $\mathcal R$, as
in Figure~\ref{fig:regions}, but the proof does not use this observation. We do not pursue the
largest region on which the variable weight is coercive.
\end{remark}

\section{Admissibility and exclusion of nonzero frequencies}\label{sec:admissible}
We now justify the boundary conditions used in the energy identity and
prove the nonzero-frequency exclusion used in Theorem~\ref{thm:main}.

\begin{lemma}[Vanishing of the energy flux]\label{lem:boundary}
Suppose $(\Gamma,\nu)\in\mathcal R$. Let $\lambda\ne0$ with $\Rea\lambda\ge0$.
With $\omega$ as in Proposition~\ref{prop:coercive}, an admissible
solution has finite energy $E_\omega$, all terms in $Q_\omega$ are
integrable, and the boundary flux in Appendix~\ref{app:identity} and the
boundary terms of Proposition~\ref{prop:energyform} tend to zero at both
ends.
\end{lemma}
\begin{proof}
At the limiting upstream state, the frozen thermal equation is
$\lambda\eps+\eps'=\nu\eps''$, and the momentum equation is
$\lambda u+u'+\Gamma\eps'=u''$. The two incoming spatial exponents are
\begin{equation}\label{eq:incomingroots}
 \mu_u^+=\frac{1+\sqrt{1+4\lambda}}2,\qquad
 \mu_e^+=\frac{1+\sqrt{1+4\nu\lambda}}{2\nu},
\end{equation}
with the square roots chosen to have positive real part. For
$\Rea\lambda\ge0$ their real parts are at least $1$ and $1/\nu$.
The remaining exponents consist of the two opposite parabolic branches
and the transport exponent $-\lambda$. None belongs to the continued
incoming space. In particular, the transport exponent remains excluded
when it lies on the imaginary axis.

Exponential convergence of the coefficients and the asymptotic
conjugation in \cite[Lemma 3.6]{HLZ} imply incoming decay
$O(e^{\omega_0x})$, including derivatives, for every
$0<\omega_0<d_-:=\min\{1,1/\nu\}$. By Lemma~\ref{lem:slopestrip},
$a'/a=d\to d_-$ as $x\to-\infty$. Hence for every $q>d_-$ there is $x_0$
with $a(x)\ge a(x_0)e^{q(x-x_0)}$ for $x\le x_0$. The coefficients $c,d$
and their derivatives are bounded by the profile equations and the slope
enclosure. In particular $\Gamma c=s-d<1$, so $c<1/\Gamma$.

Suppose first that $\omega$ is bounded above and away from zero with bounded
logarithmic derivative. Choosing $d_-<q<2d_-$ and
$q/2<\omega_0<d_-$ shows that $u/\sqrt a$, $h/\sqrt a$, and their first
derivatives decay exponentially upstream, and then so do $\xi$, $\xi'$, and
$\Phi$. This applies to $\omega=\tht$, and to the weight
\eqref{eq:upper-weight} when $\nu>\Gamma+1$. In that case $\rho>\Gamma$ gives
$c>d$, hence $c\ge s/(\Gamma+1)$, so
\[
 \varpi=a+\frac43c\ge a_*(1-s)+\frac{4s}{3(\Gamma+1)}
 \ge\min\Bigl\{a_*,\frac4{3(\Gamma+1)}\Bigr\},
\]
while $a\le a_*$ and $c<1/\Gamma$ bound $\varpi$ above, and
$\omega'/\omega=2\varpi'/\varpi$ is bounded.

Suppose next that $2/3<\nu<1$, so that $d_-=1$. Then $c\to0$ upstream and the
weight \eqref{eq:upper-weight} degenerates, at the rate obtained below. The quantity
$w=1-z$ satisfies $w'=-(1-s)f_0(z)-sf_1(z)$. Near the upstream state
$f_0(z)>0$, since $z>r$, and $-f_1(z)=w(1/\nu-z)$, so
$w'\le(1/\nu-1+w)w$. Integrating backward, for every $\tau>0$ there are
$C>0$ and $x_0$ with $w(x)\ge Ce^{(1/\nu-1+\tau)x}$ for $x\le x_0$. The same
bound holds for $c=sw/\Gamma$ and for $\varpi\ge4c/3$. Therefore
\[
 |\xi|\le C|h|\,e^{-(q/2+1/\nu-1+\tau)x}\qquad(x\le x_0).
\]
Since $\nu>2/3$, one can choose $q>1$ and $\tau>0$ so close to $1$ and
$0$, and $\omega_0<1$ so close to $1$, that
$\omega_0>q/2+1/\nu-1+\tau$. Then $\xi$ decays exponentially upstream,
and so do $\xi'$ and $\Phi$. The logarithmic derivative
$\omega'/\omega=2\varpi'/\varpi$ remains bounded, because
$c'/c=\hat v/\nu-d-da/c$ and $a/c\to0$ when $\nu>1/2$. All coefficients of the
normalized form \eqref{eq:Qomega} are bounded. In the variables
$(\Phi,d\,\xi)$ they are bounded combinations of $\rho$, $\vartheta$, and
$a/d$ by \eqref{eq:matrix-decomposition}, and $d\to1$ upstream, so the
coefficients in $(\Phi,\xi)$ are bounded as well.

In every case each upstream energy term is integrable, and every term of the
flux in the appendix and every boundary term of
Proposition~\ref{prop:energyform} vanishes.

At the downstream state, $a\to a_*>0$ and $\hat v\to v_*>0$. For nonzero
$\lambda$ in the closed right half-plane, the outgoing modes decay
exponentially. This follows from the downstream splitting in
\cite[Lemma 3.3]{HLZ}; a direct exclusion of imaginary spatial roots is
also given below. The weight, the coefficients in the energy form, and the flux
coefficients are bounded downstream. Thus all downstream integrals converge and the flux
again vanishes.
\end{proof}

We check directly that the downstream frozen system has no purely imaginary
spatial exponent for $\lambda\ne0$, $\Rea\lambda\ge0$. The frozen equations
have the form
\[
 \lambda w+\mathcal A_+w'=B_+w'',\qquad
 \mathcal A_+=\begin{pmatrix}
 1&-1&0\\-a_*/v_*&1&\Gamma/v_*\\0&a_*&1
 \end{pmatrix},\qquad
 B_+=\operatorname{diag}(0,1/v_*,\nu/v_*).
\]
The positive diagonal matrix
\[
 S_+=\operatorname{diag}\left(a_*/v_*,1,\Gamma/(a_*v_*)\right)
\]
symmetrizes $\mathcal A_+$. A putative spatial mode $w=e^{i\zeta x}y$ with real
$\zeta\ne0$ would therefore satisfy
\[
 (\Rea\lambda)y^*S_+y
 +\zeta^2\left(\frac{|y_u|^2}{v_*}
           +\frac{\Gamma\nu|y_\eps|^2}{a_*v_*^2}\right)=0.
\]
Hence $y_u=y_\eps=0$, and the momentum equation forces $y_v=0$.
For $\zeta=0$, the condition is $\lambda y=0$, again impossible for a nonzero
mode when $\lambda\ne0$. Thus the downstream frozen system has no
imaginary spatial exponent at the frequencies under consideration.

\begin{proposition}[Exclusion of nonzero frequencies]\label{prop:nonzero}
Suppose $(\Gamma,\nu)\in\mathcal R$. Then
\[
 D_*(\lambda;\Gamma,\nu)\ne0
 \qquad(\Rea\lambda\ge0,\ \lambda\ne0).
\]
\end{proposition}
\begin{proof}
Integrating the pointwise identity \eqref{eq:pointwise}, Lemma~\ref{lem:boundary} justifies the identity
$0=(\Rea\lambda)E_\omega+Q_\omega$ for every admissible solution. The energy
$E_\omega$ is nonnegative, so $Q_\omega\le0$, and
Proposition~\ref{prop:coercive} forces
$u=h=0$, and therefore $\eps=0$. The momentum equation gives $v'=0$, and the mass equation
gives $\lambda v=0$. As $\lambda\ne0$, the entire solution is trivial.
\end{proof}

\section{Zero frequency and continued outgoing limits}\label{sec:zero}
The continued outgoing condition at zero permits bounded modes, not
only modes decaying at $+\infty$. The following argument handles that
condition directly and does not require $\nu=\Gamma+1$.

\begin{proposition}[The zero-frequency condition]\label{prop:zero}
For every $\Gamma,\nu>0$, one has $D_*(0;\Gamma,\nu)\ne0$.
\end{proposition}
\begin{proof}
Suppose there is an admissible solution at $\lambda=0$. It decays upstream, is bounded downstream, and has derivatives decaying downstream. The first equation gives $u=v$. Set
\[
 w_1=u',\qquad w_2=\eps'+\hat v'u.
\]
Then
\begin{equation}\label{eq:variational}
 \binom{w_1}{w_2}'=
 \begin{pmatrix}
 2\hat v-1&\Gamma\\ \hat va/\nu+\hat e'/\hat v&\hat v/\nu
 \end{pmatrix}\binom{w_1}{w_2}.
\end{equation}
Indeed, substituting $u=v$ in \eqref{eq:ev2} gives
$w_1'=(2\hat v-1)w_1+\Gamma w_2$. Substitution in
\eqref{eq:ev3}, followed by differentiation of $w_2$, gives the second
row. The matrix in \eqref{eq:variational} is the Jacobian of the
profile vector field, and its downstream determinant is negative.
The solution $(w_1,w_2)$ decays downstream, because the derivatives of
an admissible zero-frequency solution decay and $\hat v'\to0$.
The downstream decaying solution space of this two-dimensional
variational equation is one-dimensional. It is spanned by the nonzero
profile tangent $(\hat v',\hat e')$. Integrating and using upstream decay gives,
for a constant $A_0$,
\[
 (w_1,w_2)=A_0(\hat v',\hat e'),\qquad
 u=v=A_0(\hat v-1),\qquad
 \eps=A_0\left(\hat e-\frac{(\hat v-1)^2}{2}\right).
\]
The downstream limit is $A_0j_*$, with $j_*=(-a_*,-a_*,0)^T$.

The downstream first-order flux matrix and sound speed are
\[
 \mathcal A_+=\begin{pmatrix}
 1&-1&0\\ -a_*/v_*&1&\Gamma/v_*\\0&a_*&1
 \end{pmatrix},\qquad c_+^2=\frac{2(\Gamma+1)}\Gamma>1.
\]
The characteristic speeds are $1-c_+$, $1$, and $1+c_+$. They are simple
and nonzero. The spatial exponents $\mu$ of the frozen downstream system
$\lambda w+\mathcal A_+w'=B_+w''$, with $B_+=\operatorname{diag}(0,1/v_*,\nu/v_*)$
as in Section~\ref{sec:admissible}, solve
$\det(\lambda+\mu\mathcal A_+-\mu^2B_+)=0$. At $\lambda=0$ there are three
slow exponents $\mu=0$ and two fast exponents, the roots of
\[
 \det(\mathcal A_+-\mu B_+)
 =\frac{\nu}{v_*^2}\mu^2-\Bigl(\frac{1+\nu}{v_*}-\frac{\nu a_*}{v_*^2}\Bigr)\mu
 +1-c_+^2 .
\]
The leading coefficient is positive and the constant term is negative, so
exactly one fast exponent is negative and none vanishes. Since the
characteristic speeds are simple and nonzero, the slow exponents are
analytic near $\lambda=0$, with $\mu=-\lambda/\alpha+O(\lambda^2)$ for each
characteristic speed $\alpha$, and for small $\lambda>0$ they decay at
$+\infty$ exactly when $\alpha>0$. The outgoing space continued from small
positive $\lambda$ to $\lambda=0$ is therefore spanned by the two slow
eigendirections with positive characteristic speed and the fast stable
direction. Its solutions have bounded limits in the positive-characteristic
subspace of $\mathcal A_+$, and the fast mode tends to zero.

The negative-speed left covector
\[
 \ell_-^T=(a_*/v_*,\ c_+,\ -\Gamma/v_*)
\]
annihilates every permitted bounded outgoing limit, but
\[
 \ell_-^Tj_*=-a_*\left(\frac{a_*}{v_*}+c_+\right)<0.
\]
Hence $A_0=0$, and the solution is trivial. There is no nontrivial
intersection of the continued Evans spaces at zero.
\end{proof}

Propositions~\ref{prop:nonzero} and~\ref{prop:zero} prove
Theorem~\ref{thm:main}.

\section{Transport perturbations and sufficiently strong finite shocks}\label{sec:finite}
We use the parameter dependence of the profile and spectral equations to
pass from the limiting profiles of Theorem~\ref{thm:main} to sufficiently
strong finite shocks. Write
\begin{equation}\label{eq:varsigma}
 \varsigma=v_+-v_*(\Gamma),\qquad
 \mathcal D(\lambda;\Gamma,\nu,\varsigma)
 =D(\lambda;\Gamma,\nu,v_*(\Gamma)+\varsigma).
\end{equation}
The use of $\varsigma$ places the strong-shock limit at the fixed parameter value
$\varsigma=0$.

\begin{lemma}[Uniform high-frequency exclusion]\label{lem:highfrequency}
Let $\mathcal K$ be a compact subset of $\{\Gamma>0,\nu>0\}$.
There are $\varsigma_0>0$ and $\Lambda<\infty$ such that
\[
 \mathcal D(\lambda;\Gamma,\nu,\varsigma)\ne0
 \quad\text{if }(\Gamma,\nu)\in\mathcal K,\quad
 0\le\varsigma\le\varsigma_0,\quad
 \Rea\lambda\ge0,\quad |\lambda|>\Lambda.
\]
The amplitude bound can be chosen so that $v_*(\Gamma)+\varsigma<1$
throughout this parameter set.
\end{lemma}
\begin{proof}
Choose $\varsigma_0<\min_{\mathcal K}a_*(\Gamma)$.
The profile bounds and exponential convergence of
\cite[Lemma 2.3]{HLZ} are uniform on the resulting compact parameter
set. In particular, $\hat v$ is bounded away from zero, and the profile and
finitely many derivatives entering the spectral coefficients are bounded.

Proposition~4.1 of \cite{HLZ} states that, for $\gamma>1$ and
$v_*<v_+\le1$, eigenvalues with nonnegative real part satisfy
$|\lambda|\le\Lambda$ for every $\Lambda$ satisfying the inequality
\cite[(4.1)]{HLZ}. The coefficients $|F^*_{kl}|$ in that inequality, defined
in \cite[(4.2)]{HLZ}, are built from the profile, finitely many of its
derivatives, and $\hat v^{-1/2}$. By the uniform profile bounds above they
are bounded uniformly on the compact parameter set, including
$\varsigma=0$, so a single sufficiently large $\Lambda$ satisfies (4.1)
throughout. The proof of \cite[Proposition 4.1]{HLZ} uses only these bounds,
and \cite[Table 1]{HLZ} evaluates the resulting radius at $v_+=v_*$, so the
same exclusion holds at $\varsigma=0$.
\end{proof}

\begin{proof}[Proof of Corollary~\ref{cor:finite}]
The set $\mathcal K$ is a compact subset of the positive parameter
quadrant. Use $\varsigma_0,\Lambda$ supplied by
Lemma~\ref{lem:highfrequency} for this set, and put
\[
 K_\Lambda=\{\lambda\in\C:\Rea\lambda\ge0,\ |\lambda|\le\Lambda\}.
\]
Proposition~3.8 of \cite{HLZ} states that the Evans function is analytic
in $\lambda$ and continuous in $(\Gamma,\nu,v_+)$ on $\Rea\lambda\ge0$ for
$\Gamma>0$, $\nu>0$, and $1>v_+\ge v_*$, a range that includes the
strong-shock limit and $\lambda=0$, and \cite[Remark 3.9]{HLZ} records the
resulting convergence as $v_+\to v_*$, uniformly on compact subsets of
$\{\Rea\lambda\ge0\}$. Thus $\mathcal D$ is jointly continuous through
$\varsigma=0$. By Theorem~\ref{thm:main} and compactness,
\[
 m_*:=\min_{(\Gamma,\nu)\in\mathcal K,\ \lambda\in K_\Lambda}
 |D_*(\lambda;\Gamma,\nu)|>0.
\]
Uniform continuity on the compact set
$\mathcal K\times[0,\varsigma_0]\times K_\Lambda$ then yields
$0<\delta_{\mathcal K}\le\varsigma_0$ such that
$|\mathcal D(\lambda;\Gamma,\nu,\varsigma)-D_*(\lambda;\Gamma,\nu)|<m_*/2$
for $\lambda\in K_\Lambda$, $(\Gamma,\nu)\in\mathcal K$, and
$0\le\varsigma<\delta_{\mathcal K}$. Thus $\mathcal D$ is nonzero on
$K_\Lambda$, and Lemma~\ref{lem:highfrequency} excludes zeros outside that
set.
\end{proof}

The shock-strength assertion can also be expressed in terms of the
upstream Mach number. In the present normalization, its relation to
$\varsigma$ is \cite[Appendix D]{HLZ}
\begin{equation}\label{eq:Mach}
 M^2=\frac{2}{(\Gamma+2)\varsigma}.
\end{equation}
Consequently, for every compact $\mathcal K\subset\mathcal R$, there is a
finite $M_{\mathcal K}$ such that the physical shocks with
$(\Gamma,\nu)\in\mathcal K$ and $M>M_{\mathcal K}$ satisfy the integrated
Evans stability condition. These Mach thresholds are existential, just as
$\delta_{\mathcal K}$ is.

\section{Formal verification}\label{sec:formal}
The arguments of this paper have been checked in the Lean~4 proof assistant
\cite{Lean4} using the Mathlib library \cite{Mathlib}. The formalization
\LeanArchive. It covers each step of the proof, from the cancellation of the
thermal mass coupling in Section~\ref{sec:thermal} to the passage to strong
finite shocks in Section~\ref{sec:finite}, and it also checks that the gases
of Appendix~\ref{app:physical} lie in the region $\mathcal R$. The checked
statements depend only on the standard axioms of Lean.

Results that are not yet available in Mathlib enter the formalization as
hypotheses. These are the existence of the limiting profile and its endstate
limits, the saddle structure of the downstream rest point, the exponential
decay of admissible solutions, the structure of the problem at zero
frequency, and the properties of the integrated Evans function established
in \cite{HLZ}.

\appendix
\section{Pointwise identity}\label{app:identity}
Let $R_1=\lambda v+v'-u'$ and let $R_2,R_h$ be the residuals in
\eqref{eq:r2}--\eqref{eq:rh}. Let $\omega>0$ be a weight as in
Section~\ref{sec:energy}, and write
\[
 M_\omega=\frac1{a\omega},\qquad
 \mathcal J=M_\omega\Bigl[F+(\nu-1)c\,\frac{\omega'}{\omega}-\Gamma\omega\Bigr],
\]
with $b_h$ and $F$ as in \eqref{eq:F}. The energy form is
\begin{equation}\label{eq:rawform}
\begin{aligned}
 Q_\omega=\int_\R\Bigl[&\Bigl(1-\frac\Gamma2\Bigr)a\,a'|\phi|^2+|\Phi|^2
 +\nu M_\omega|h'|^2-(\nu-1)cM_\omega\sqrt a\Rea(\Phi\bar h')\\
 &+\mathcal J\sqrt a\Rea(\Phi\bar h)
 -\frac{\nu M_\omega''+(b_hM_\omega)'}{2}|h|^2\Bigr]\dd .
\end{aligned}
\end{equation}
Define the flux
\begin{align*}
 \mathcal F={}&\frac12|v|^2-\Rea(u\bar v)
 -\Rea\Bigl(\frac{\Phi\bar u}{\sqrt a}\Bigr)+\frac\Gamma4|u|^2
 +\Gamma\Rea\Bigl(\frac{h\bar u}a\Bigr)\\
 &-\nu M_\omega\Rea(h'\bar h)+\frac{\nu M_\omega'+b_hM_\omega}{2}|h|^2
 +(\nu-1)cM_\omega\sqrt a\Rea(\Phi\bar h).
\end{align*}
If $E_{\rm dens}=|v|^2+(\hat v/a)|u|^2+\hat vM_\omega|h|^2$ and $Q_{\rm dens}$
is the integrand in \eqref{eq:rawform}, then
\begin{equation}\label{eq:pointwise}
 \Rea\Bigl(R_1\bar v+\frac{\hat v}aR_2\bar u
 +\frac{\hat v}{a\omega}R_h\bar h\Bigr)
 =(\Rea\lambda)E_{\rm dens}+Q_{\rm dens}+\mathcal F'.
\end{equation}
This identity uses the limiting profile relations but does not assume that
the spectral residuals vanish. Integrating it on a finite interval gives the
energy balance with the two endpoint fluxes displayed explicitly.

The derivation is direct. The momentum terms follow from the factorization
\eqref{eq:momfactor}, as in the proof of Proposition~\ref{prop:energyform}.
In the thermal residual \eqref{eq:rhU}, one integration by parts in $-\nu h''$
and one in $(\nu-1)c(\sqrt a\,\Phi)'$ give the diffusion, the coupling with
$\bar h'$, and the potential, while one integration by parts in the momentum
term $\Gamma h'$ gives the remaining part of the coupling with $\bar h$. The
boundary terms of these steps assemble into $\mathcal F$, and the identity
$a\omega\,\mathcal J=F+(\nu-1)c\,\omega'/\omega-\Gamma\omega$ collects the
coefficient of $\sqrt a\Rea(\Phi\bar h)$.

\section{Continuation of the limiting Evans function}\label{sec:continuation}
The main text proves limiting stability directly by energy estimates
and a zero-frequency argument. For completeness, this appendix records
an alternative Evans-function continuation principle. It isolates the
only possible boundary crossing along a compact parameter path.

\begin{proposition}[Evans continuation along a parameter path]\label{prop:homotopy}
Let $p:[0,1]\to(0,\infty)^2$ be continuous. Suppose $D_*(\lambda;p(0))$
is nonzero on the closed right half-plane, and
\[
 D_*(i\tau;p(t))\ne0\qquad(t\in[0,1],\ \tau\in\R\setminus\{0\}).
\]
Then $D_*(\lambda;p(t))$ is nonzero for all $t\in[0,1]$ and
$\Rea\lambda\ge0$.
\end{proposition}
\begin{proof}
The image of the path is a compact subset of the positive parameter
quadrant. The uniform high-frequency argument in
Lemma~\ref{lem:highfrequency} supplies a common radius $\Lambda$ beyond which
no nonstable Evans zero is possible. Proposition~\ref{prop:zero} and
continuity exclude a common neighborhood of $\lambda=0$ along the path.
On the remaining imaginary boundary the Evans function is nonzero by
hypothesis. With bases chosen continuously along the path and analytically
in $\lambda$, the winding number on the boundary of a sufficiently large
right half-disk is therefore a continuous integer-valued function of $t$.
It is the number of zeros in the open right half-disk, counted with
multiplicity, and is zero at $t=0$. It remains zero for all $t$.
Together with the boundary exclusions, this proves the result.
\end{proof}

\begin{remark}[Paths in the region]
Let $(\Gamma,\nu)\in\mathcal R$ and consider the path
$p(t)=\bigl(\Gamma,\Gamma+1+t[\nu-(\Gamma+1)]\bigr)$, which begins on the
reference line. It remains in $\mathcal R$ when $\nu\le\max\{9/5,\Gamma+1\}$,
since $\Gamma_E$ vanishes on $(0,9/5]$ and $\Gamma_E(\nu')<\nu'-1$, and on all of
$\mathcal R$ if $\Gamma_E$ is increasing, as numerical evaluation indicates.
For such points
Proposition~\ref{prop:nonzero} excludes every nonzero imaginary zero
along the path, so Proposition~\ref{prop:homotopy} transports stability
from the reference line to the endpoint once it is known on the line.
\end{remark}

\section{Physical transport ratios}\label{app:physical}
Under Stokes' hypothesis, \eqref{eq:Prandtl} gives
$\nu=3\gamma/(4\operatorname{Pr})$. For ordinary gases the Prandtl number
varies little, from the kinetic-theory value $\operatorname{Pr}=2/3$ for
monatomic gases to about one for water vapor. Eucken's estimate
$\operatorname{Pr}=4\gamma/(9\gamma-5)$ gives $\nu=(27\gamma-15)/16$, which
is $15/8$ for monatomic gases and about $1.4$ for diatomic gases.
Table~\ref{tab:physical} compares representative values with the upper
end $\nu_E(\Gamma)$ of the region and with the limit $\nu_D(\Gamma)$ of the
constant weight from Remark~\ref{rem:sharp}.

\begin{table}[ht]
\centering
\begin{tabular}{lccccc}
 Gas & $\gamma$ & $\operatorname{Pr}$ & $\nu=3\gamma/(4\operatorname{Pr})$ & $\nu_E(\Gamma)$ & $\nu_D(\Gamma)$\\
 \hline
 Monatomic gases & $5/3$ & $2/3$ & $1.88$ & $2.94$ & $2.84$\\
 Hydrogen & $1.41$ & $0.70$ & $1.51$ & $2.49$ & $2.29$\\
 Air & $1.40$ & $0.71$ & $1.48$ & $2.47$ & $2.27$\\
 Nitrogen & $1.40$ & $0.72$ & $1.46$ & $2.47$ & $2.27$\\
 Carbon dioxide & $1.29$ & $0.77$ & $1.26$ & $2.28$ & $2.01$\\
 Water vapor & $1.33$ & $1.0$ & $1.00$ & $2.35$ & $2.11$
\end{tabular}
\caption{Representative transport ratios near room temperature. The
Prandtl numbers are rounded values from standard property tables
\cite{Incropera}. The column $\nu_E(\Gamma)$ is the largest $\nu$ in
$\mathcal R$ at the given $\gamma$, and $\nu_D(\Gamma)$ is the corresponding
value for the constant weight. Both are computed numerically for the value of
$\gamma$ shown and rounded down.}
\label{tab:physical}
\end{table}

Every entry lies inside the region, with $\nu$ at least $0.75$ below the
computed value of $\nu_E$. Membership can be checked by hand. Every entry except the
monatomic one has $2/3<\nu\le9/5$. For monatomic gases $\Gamma=2/3$ and
$\nu=15/8$, so $\Gamma<\rho_r<\nu-1$ by \eqref{eq:sharp-slope}, and
Lemma~\ref{lem:det-monotone} gives
$\Pi_\nu(\rho_r)>\Pi_\nu(2/3)=505535/256>0$, which is the endpoint condition
\eqref{eq:endpoint-criterion}. Water vapor, near $\nu=1$, lies in the strip
$2/3<\nu\le9/5$. Under Stokes' hypothesis the
lower end $\nu>2/3$ of the region is $\operatorname{Pr}<9\gamma/8$, which holds
for every common gas. A bulk viscosity
increases the effective longitudinal viscosity and so lowers $\nu$ further.
Gases outside the region at the upper end include mixtures of light and
heavy noble gases, such as helium and xenon, whose Prandtl numbers can fall
to about $0.2$ to $0.3$, which gives $\nu$ between about $4$ and $6$.
These statements concern the constant-coefficient model of
Section~\ref{sec:introduction}.

\section*{Declarations}
\noindent\textbf{Code availability.}
The Lean formalization described in Section~\ref{sec:formal}, together with a
notebook that rebuilds and checks it, is archived at Zenodo \cite{HumLean}.

\medskip
\noindent\textbf{Use of generative AI.}
During the preparation of this work the author used ChatGPT (OpenAI) and
Claude (Anthropic) to assist with the research, the formal verification, and
the preparation of the manuscript. The author reviewed and edited the
content and takes full responsibility for it.

\medskip
\noindent\textbf{Competing interests.} \ConflictStatement

\end{document}